%% file: main.tex
\documentclass{article}

\input{all_commands}

\usepackage[verbose=true,letterpaper]{geometry}
\AtBeginDocument{
  \newgeometry{
    textheight=9in,
    textwidth=5.5in,
    top=1in,
    headheight=12pt,
    headsep=25pt,
    footskip=30pt
  }
}

\author{%
    Slavom\'ir Hanzely\thanks{Equal contributions}\,
    \thanks{Mohamed bin Zayed University of Artificial Intelligence, United Arab Emirates},
    \thanks{King Abdullah University of Science and Technology, Thuwal, Saudi Arabia}\\
	\texttt{shanzely@gmail.com}
    \and
    Dmitry Kamzolov\footnotemark[1]\,\,\footnotemark[2]\\
    \texttt{kamzolov.opt@gmail.com}
    \and 
    Dmitry Pasechnyuk\footnotemark[2]\,\,\,\thanks{Moscow Institute of Physics and Technology, Russia}\,\,\,\thanks{ISP RAS Research Center for Trusted Artificial Intelligence, Russia}\\
    \texttt{dmivilensky1@gmail}
    \and 
    Alexander Gasnikov\footnotemark[4]\,\,\,\footnotemark[5]\,\,\,\thanks{National Research University Higher School of Economics, Russia}\\
    \texttt{gasnikov@yandex.ru}
    \and
    Peter Richt\'arik\footnotemark[3]\\
    \texttt{richtarik@gmail.com}
    \and
    Martin Tak\'a\v{c}\footnotemark[2]\\
    \texttt{takac.MT@gmail.com}
    }
    \date{}
    \title{\textbf{A Damped Newton Method Achieves Global $\cO\left(\frac 1 {k^2}\right)$  and Local Quadratic  Convergence Rate}}

\begin{document}
	\maketitle
	\begin{abstract}
		In this paper, we present the first stepsize schedule for Newton method resulting in fast global and local convergence guarantees. In particular, a) we prove an $\cO\left( \frac 1 {k^2} \right)$ global rate, which matches the state-of-the-art global rate of cubically regularized  Newton method of Polyak and Nesterov (2006) and of regularized Newton method of Mishchenko (2021) and Doikov and Nesterov (2021), b) we prove a local quadratic rate, which matches the best-known local rate of second-order methods, and c) our stepsize formula is simple, explicit, and does not require solving any subproblem. Our convergence proofs hold under affine-invariance assumptions closely related to the notion of self-concordance. Finally, our method has competitive performance when compared to existing baselines, which share the same fast global convergence guarantees.
	\end{abstract}
	
% 	\tableofcontents	
	
	\section{Introduction} \label{sec:intro}
	%\vspace{-0.2cm}
Second-order optimization methods are the backbone of much of industrial and scientific computing. With origins that can be tracked back several centuries to the pioneering works of  Newton~\citep{Newton}, Raphson~\citep{Raphson} and Simpson~\citep{Simpson}, they were extensively studied, generalized, modified, and improved in the last century \citep{kantorovich1948functional, more1978levenberg, Griewank-cubic-1981}. For a review of the historical development of the classical Newton-Raphson method, we refer the reader to the work of \citet{ypma1995historical}. The number of extensions and applications of second-order optimization methods is enormous; for example, the survey of \citet{conn2000trust} on trust-region and quasi-Newton methods  cited over a thousand papers. 
%\vspace{-0.2cm}

\subsection{Second-order methods and modern machine learning}		
%\vspace{-0.2cm}
Despite the rich history of the field, research on second-order methods has been flourishing up to this day. Some of the most recent development in the area was motivated by the needs of modern machine learning. 
Data-oriented machine learning depends on large datasets (both in number of features and number of datapoints), which are often stored in distributed/decentalized fashion. Consequently, there is a need for scalable algorithms.

To tackle large number of features, \citet{SDNA, RSN, RBCN} and \citet{SSCN} proposed variants of Newton method operating in random low-dimensional subspaces. On the other hand, \citet{NewtonSketch,Roosta-MAPR} and \citet{SN2019} developed subsampled Newton methods for solving empirical risk minimization (ERM) problems with large training datasets. Additionally, \citet{Bordes2009a,OnlineBFGS,SBFGS,SQN2016} and \citet{RBFGS2020} proposed stochastic variants of  quasi-Newton methods.
To tackle non-centralized nature of datasets, \citet{DANE, AIDE, GIANT} and \citet{DINGO2019} considered distributed variants of Newton method, with improvements under various data/function similarity assumptions. \citet{NL2021, FedNL, BL2022, Newton-3PC} and \citet{agafonov2022flecs} developed communication-efficient distributed variants of Newton method using the idea of communication compression and error compensation, without the need for any similarity assumptions. 

We highlight two main research directions throughout of history of second-order methods: globally convergent methods under additional second-order smoothness \eqref{eq:L2-smooth} \citep{nesterov2006cubic} and local methods for self-concordant problems \eqref{eq:self-concordance} \citep{nesterov1989self}. 
	Former approach lead to various improvements such as acceleration \citet{nesterov2008accelerating, monteiro2013accelerated}, usage of inexact information \citet{ghadimi2017second, agafonov2020inexact}, generalization to tensor methods and their acceleration \citet{nesterov2021implementable, gasnikov2019near, kovalev2022first}, superfast second-order methods under higher smoothness \citet{nesterov2021superfast, nesterov2021inexact, kamzolov2020near}.
	Latter approach was a breakthrough in 1990s, it lead to interior-point methods. Summary of the results can be found in books \citet{nesterov1994interior}, \citet{nesterov2018lectures}. This direction is still popular up to this day \citet{dvurechensky2018global, hildebrand2020optimal, doikov2022affine, nesterov2022set}.

    As easy-to-scale alternative to second-order methods, first-order algorithms attracted a lot of attention. Many of their aspects have been explored, including strong results in variance reduction (\citet{roux2012stochastic},\citet{gower2020variance},
	\cite{johnson2013accelerating,nguyen2021inexact,nguyen2017sarah}), preconditioning \cite{jahani2022doubly}, acceleration (\citet{nesterov2013introductory},
	\citet{aspremont2021acceleration}) distributed/federated computation (\citet{konecny2016federated,chen2022distributed,berahas2016multi,takavc2015distributed,richtarik2016distributed}, \citet{kairouz2021advances}), and  decentralized computation \citep{koloskova2020unified,sadiev2022decentralized,borodich2021decentralized}.
 However, the convergence of first-order methods always depends on the conditioning of the underlying problem. Improving conditioning is fundamentally impossible without using higher-order information.
	Removing this conditioning dependence is possible by incorporating information about the Hessian. This results in second-order methods. Their most compelling advantage is that they can converge extremely quickly, usually in just a few iterations.
	\vspace{-0.2cm}
	
	\subsection{Newton method: benefits and limitations}
	\vspace{-0.2cm}
	One of the most famous algorithms in optimization, Newton method, takes iterates of form 
	\begin{equation} \label{eq:newton_step}
	x_{k+1} = x_k - \left[\nabla^2 f(x_k) \right]^{-1} \nabla f(x_k).
	\end{equation}
	Its iterates satisfy the recursion $\norm{\g(x_{k+1})} \leq c \norms
	{\g(x_k)}$ (for a constant $c>0$), which means that Newton method converges locally quadratically. 
	However, convergence of Newton method is limited to only to the neighborhood of the solution.  
	It is well-known that when initialized far from optimum, Newton can diverge, both in theory and practice (\citet{jarre2016simple}, \citet{mascarenhas2007divergence}). We can explain intuition why this happens.
	Update rule of Newton \eqref{eq:newton_step} was chosen to minimize right hand side of Taylor approximation
	\begin{equation} \label{eq:newton_approximation}
	\gboxeq{ f(y) \approx Q_{f}(y;x) \eqdef f(x) + \la\nabla f(x),y-x\ra + \frac{1}{2} \la \nabla^2 f(x)(y-x),y-x\ra.}
	\end{equation}
	The main problem is that Taylor approximation is not an upper bound, and therefore, global convergence of Newton method is not guaranteed.
	
	\subsection{Towards a fast globally convergent Newton method}
	Even though second-order algorithms with superlinear local convergence rates are very common, global convergence guarantees of any form are surprisingly rare. Many papers proposed globalization strategies, essentially all of them require some combination of the following: line-search, trust regions, damping/truncation, regularization. 
	Some popular globalization strategies show non-increase of functional value during the training. However, this turned out to be insufficient for convergence to the optimum.  \citet{jarre2016simple}, \citet{mascarenhas2007divergence} designed simple functions (strictly convex with compact level sets) so that Newton method with Armijo stepsizes does not converge to the optimum. 
	To this day, virtually all known global convergence guarantees are for regularized Newton methods with, which can be written as
	\begin{equation} \label{eq:LM_update}
	x_{k+1} = x_k - \alpha_k \left( \h(x_k) + \lambda_k \mI \right)^{-1} \g(x_k),
	\end{equation}
	where $\lambda_k \geq 0$.
	Parameter $\lambda_k$ is also known as Levenberg-Marquardt regularization \citep{more1978levenberg}, which was first introduced for a nonlinear least-squares objective. For simplicity, we disregard differences in the objectives for the literature comparison.
	Motivation behind update \eqref{eq:LM_update} is to replace Taylor approximation in \eqref{eq:newton_approximation} by an upper bound. 
	The first method with proven global convergence rate $\cO\left(k^{-2}\right)$ is Cubic Newton method \citep{nesterov2006cubic} for function $f$ with Lipschitz-continuous Hessian,
	\begin{equation} \label{eq:L2-smooth}
	\gboxeq{ \| \nabla^2 f(x) - \nabla^2 f(y)\|_2 \leq L_2\|x-y\|_2.}
	\end{equation} 
	Under this condition, one can upper bound of Taylor approximation \cref{eq:newton_approximation} as
	\begin{equation} \label{eq:cubic_newton_approximation}
	f(y) \leq Q_{f}(y;x) +\tfrac{L_2}{6}\|y-x\|_2^3.
	\end{equation}
	Next iterate of Cubic Newton can be written as a minimizer of right hand side of \eqref{eq:cubic_newton_approximation}\footnote{\label{ft:E}Where $\mathbb{E}$ a $d$-dimensional Euclidean space, defined in \Cref{ssec:notation}.}
	\begin{equation} \label{eq:cubic_newton_step}
	x_{k+1} = \argmin_{y\in \bbE} \lb Q_{f}(y;x_k) +\tfrac{L_2}{6}\|y-x_k\|_2^3 \rb.
	\end{equation}
	For our newly-proposed algorithm \ain{} (\Cref{alg:ain}), we are using almost identical step\footref{ft:E} \footnote{\label{ft:Lf}Function $f$ is $\Lsemi$-semi-strongly self-concordant (\Cref{def:semi-self-concordance}). Instead of $\Lsemi$, we will use its upper bound $\Lalg$, $\Lalg \geq \Lsemi$.}
	\begin{equation} \label{eq:aicn_step}
	\gboxeq{x_{k+1} =  \argmin_{y\in \bbE} \left\lbrace Q_{f}(y;x_k) +\tfrac{\Lsemi}{6}\|y-x_k\|_{x_k}^3\right\rbrace.}
	\end{equation}
	The difference between the update of Cubic Newton and \ain{} is that we measure the cubic regularization term in the local Hessian norms. This seemingly negligible perturbation turned out to be of a great significance for two reasons \textbf{a)} model in \eqref{eq:aicn_step} is affine-invariant, \textbf{b)} surprisingly, the next iterate of \eqref{eq:aicn_step} lies in the direction of Newton method step and is obtainable without regularizer $\lambda_k$ (\ain{} just needs to set stepsize $\alpha_k$). We elaborate on both of these points later in the paper.

	Cubic Newton method \eqref{eq:cubic_newton_step} can be equivalently expressed in form \eqref{eq:LM_update} with $\alpha_k=1$ and $\lambda_k = L_2 \norm{x_k -x_{k+1}}$. However, since such $\lambda_k$ depends on $x_{k+1}$, resulting algorithm requires additional subroutine for solving its subproblem each iteration.
	Next work showing convergence rate of regularized Newton method \citep{polyak2009regularized} avoided implicit steps by choosing $\lambda_k\propto \norm{\g(x_k)}$. However, this came with a trade-off for slower convergence rate, $\cO\left( k^{-1/4}\right)$.
	Finally, \citet{mishchenko2021regularized} (see also the work of \citet{doikov2021optimization}) improved upon both of these works by using explicit regularization $\alpha_k=1$, $\lambda_k \propto \sqrt{L_2 \norm{\g(x_k)}}$, and proving global rate $\cO\left(k^{-2}\right)$.
	\vspace{-0.2cm}
	
	%%%%%%%%
	\section{Contributions} \label{sec:contributions}	
	\vspace{-0.2cm}
	\subsection{\ain{} as a damped Newton method}
	\vspace{-0.2cm}
	In this work, we investigate global convergence for most basic globalization strategy, stepsized Newton method without any regularizer ($\lambda_k=0$). This algorithm is also referred as damped (or truncated) Newton method; it can be written as 
	\[x_{k+1} = x_k - \alpha_k \h(x_k) ^{-1} \g(x_k).\]
	Resulting algorithm was investigated in detail as an interior-point method. \citet{nesterov2018lectures} shows quadratic local convergence for stepsizes $\alpha_1 \eqdef \frac 1 {1 + G_1}, \alpha_2 \eqdef \frac {1+ G_1} {1 + G_1 + G_1^2}$, where\footnote{Function $f$ is $\Lstandard$-self-concordant (\Cref{def:self-concordance}).}\footnote{\label{ft:G}Dual norm $\normMd{\g(x_k)} {x_k} = \<\g(x_k), \h(x_k)^{-1} \g(x_k)>$ is defined in \Cref{ssec:notation}.} $G_1 \eqdef \Lstandard \normMd {\g(x_k)} {x_k}$. 
	Our algorithm \ain{} is also damped Newton method with stepsize \gbox{$\alpha = \frac {-1 + \sqrt{1+ 2G}}{G}$,} where\footref{ft:Lf}\footref{ft:G} \gbox{$G \eqdef \Lsemi \normMd {\g(x_k)} {x_k}$. } Mentioned stepsizes $\alpha_1, \alpha_2, \alpha$ share two characteristics. Firstly, all of them depends on gradient computed in the dual norm and scaled by a smoothness constant ($G_1$ or $G$). Secondly, all of these stepsizes converge to $1$ from below (for $\hat \alpha \in \{\alpha_1, \alpha_2, \alpha\}$ holds $0< \hat \alpha \leq 1$ and $\lim_{x \rightarrow x_*} \hat \alpha = 1$). Our algorithm uses stepsize bigger by orders of magnitude (see \Cref{fig:stepsizes} in \Cref{sec:ap_experiments} for detailed comparison).
	The main difference between already established stepsizes $\alpha_1, \alpha_2$ and our stepsize $\alpha$ are resulting global convergence rates. While stepsize $\alpha_2$ does not lead to a global convergence rate, and $\alpha_1$ leads to rate $\cO \left( k^{-1/2} \right)$, our stepsize $\alpha$ leads to a significantly faster, $\cO\left(k^{-2}\right)$ rate.
	Our rate matches best known global rates for regularized Newton methods. We manage to achieve these results by carefully choosing assumptions. While rates for $\alpha_1$ and $\alpha_2$ follows from standard self-concordance, 
	 our assumptions are a consequence of a slightly stronger version of self-concordance. We will discuss this difference in detail later.
% 	\\[2pt]
	
	We summarize important properties of regularized Newton methods with fast global convergence guarantees and damped Newton methods in \Cref{tab:global_newton_like}.
	
	\begin{table*}[!t]
		\centering
		\setlength\tabcolsep{0.7pt} 
		\begin{threeparttable}[b]
			{\scriptsize
				\renewcommand\arraystretch{2.2}
				\caption{A summary of regularized Newton methods with global convergence guarantees. We consider algorithms with updates of form $x_{k+1} = x_k - \alpha_k \left( \h(x_k) + \lambda_k \mI \right)^{-1} \g(x_k)$. For simplicity of comparison, we disregard differences in objectives and assumptions. We assume $L_2$-smoothness of Hessian, $\Lsemi$-semi-strong self-concordance, convexity (\Cref{def:semi-self-concordance}), $\mu$-strong convexity locally and bounded level sets. For regularization parameter holds $\lambda_k\geq0$ and stepsize satisfy $0<\alpha_k \leq 1$. We highlight the best know rates in blue.}
				\label{tab:global_newton_like}
				\centering 
				\begin{tabular}{ccccccccc}\toprule[.1em]
					\bf Algorithm &\bf  \makecell{Regularizer \\$\lambda_k \propto$} & \bf \makecell{Stepsize\\ $\alpha_k=$} & \bf  \makecell{Affine\tnote{\color{red}(1)} \\ invariant?\\ (alg., ass., rate)} &\bf  \makecell{Avoids \\ line\\ search?} &\bf  \makecell{Global \\ convergence \\ rate} & \bf \makecell{Local \tnote{\color{red}(1)} \\convergence\\exponent} & \bf Reference \\
					\midrule
					\makecell{Newton} & $0$ & $1$ & (\cmark, \xmark, \xmark) & \cmark & \xmark & {\color{blue} $ 2$} &  \citet{kantorovich1948functional} \\
					\makecell{Newton} & $0$ & $1$ & (\cmark,\cmark,\cmark) & \cmark & \xmark & {\color{blue} $ 2$} &  \citet{nesterov1994interior} \\
					\makecell{Damped Newton B} &$0$ & $\tfrac 1 {1+G_1}$\tnote{\color{red}(4)} & (\cmark,\cmark,\cmark) & \cmark & $\cO\left( k^{-\frac 12} \right)$ & {\color{blue} $ 2$} &  \citet[$(5.1.28)$]{nesterov2018lectures} \\
					\makecell{Damped Newton C} &$0$ & $\tfrac {1 + G_1} {1+G_1 + G_1^2}$\tnote{\color{red}(4)} & (\cmark,\cmark,\cmark) & \cmark & \xmark & {\color{blue} $ 2$} &  \citet[$(5.2.1)_C$]{nesterov2018lectures} \\
					\makecell{Cubic Newton} &$L_2\norm{\xdiff} $ & $1$ & (\xmark, \xmark, \xmark) & \xmark & $\cO\left( {\color{blue}{ k^{-2}} } \right)$ & {\color{blue} $ 2$} &  \makecell{\citet{nesterov2006cubic},\\\citet{griewank1981modification},\\ \citet{doikov2021local}} \\
					\makecell{Locally Reg. Newton} &$\norm{\g(x_k)} $ & $1$ & (\xmark, \xmark, \xmark) & \cmark & \xmark & {\color{blue} $ 2$} &  \makecell{\citet{polyak2009regularized}} \\
					\makecell{Globally Reg. Newton} &$\norm{\g(x_k)} $ & $\tfrac{\mu +\norm{\g(x_k)}}{L_1}$\tnote{\color{red}(3)} & (\xmark, \xmark, \xmark) & \xmark & $\cO\left( k^{-\frac 14} \right)$ & {\color{blue} $ 2$} &  \makecell{\citet{polyak2009regularized}} \\
					\makecell{Globally Reg. Newton} &$\sqrt{L_2\norm{\g(x_k)}}$ & $1$ & (\xmark, \xmark, \xmark) & \cmark & $\cO\left( {\color{blue}{ k^{-2}} } \right)$ & $ \tfrac 32$ &  \makecell{\citet{mishchenko2021regularized} \\ \citet{doikov2021optimization}} \\
					\midrule
					\makecell{\textbf{AIC Newton}\\ (\Cref{alg:ain})} & $0$ & $\tfrac {-1 + \sqrt{1+2 G} }{G}$ \tnote{\color{red}(4)} & (\cmark, \cmark, \cmark) & \cmark & $\cO\left( {\color{blue} {k^{-2}} } \right)$ & {\color{blue} $ 2$} &   \textbf{This work} \\
					\bottomrule[.1em]
				\end{tabular}
			}
			\begin{tablenotes}
				{\scriptsize
					\item [\color{red}(1)] In triplets, we report whether algorithm, used assumptions, convergence rate are affine-invariant, respectively.
					\item [\color{red}(2)] For a Lyapunov function $\Phi^k$ and a constant $c$, we report exponent $\beta$ of $\Phi(x_{k+1}) \leq c \Phi (x_k)^\beta$.
					\item [\color{red}(3)] $f$ has $L_1$-Lipschitz continuous gradient.
					\item [\color{red}(4)] For simplicity, we denote $G_1 \eqdef \Lstandard \normMd {\g(x_k)} {x_k}$ and $G \eqdef \Lsemi \normMd {\g(x_k)} {x_k}$ (for $\Lalg \leftarrow \Lsemi$).
				}
			\end{tablenotes}
		\end{threeparttable}
		\vspace{-1cm}
	\end{table*}

	\vspace{-0.2cm}
\subsection{Summary of contributions}
\vspace{-0.2cm}
	To summarize novelty in our work, we present a novel algorithm \ain{}. Our algorithm can be interpreted in two viewpoints \textbf{a)} as a regularized Newton method (version of Cubic Newton method), \textbf{b)} as a damped Newton method. \ain{} enjoys the best properties of these two worlds: 
	\vspace{-0.2cm}
\begin{itemize}[leftmargin=*]
    \item \textbf{Fast global convergence}: \ain{} converges globally with rate $\cO\left( k^{-2} \right)$ (\Cref{thm:convergence}, \ref{th:global}), which matches state-of-the-art global rate for all regularized Newton methods. Furthermore, it is the first such rate for Damped Newton method.
	
    \item \textbf{Fast local convergence:} In addition to the fast global rate, \ain{} decreases gradient norms locally in quadratic rate (\Cref{th:local}). This result matches the best-known rates for both regularized Newton algorithms and damped Newton algorithms.
	
    \item \textbf{Simplicity:} Previous works on Newton regularizations can be viewed as a popular global-convergence fix for the Newton method. We propose an even simpler fix in the form of a stepsize schedule (\Cref{sec:alg}).
    
    \item \textbf{Implementability:} 
    Step of \ain{} depends on a smoothness constant $\Lsemi$ (\Cref{def:semi-self-concordance}). Given this constant, next iterate of \ain{} can be computed directly. 
    
    This is improvement over Cubic Newton \citep{nesterov2006cubic}, which for a given constant  $L_2$ needs to run \textbf{line-search} subroutine each iteration to solve its subproblem.
    
    \item \textbf{Improvement:} Avoiding latter subroutine yields theoretical improvements.
    If we compute matrix inverses naively, iteration cost of \ain{} is $\cO(d^3)$ (where $d$ is a dimension of the problem), which is improvement over 
    $\cO(d^3 \log \varepsilon^{-1})$ iteration cost of Cubic Newton \citep{nesterov2006cubic}.

    \item \textbf{Practical performance:} We show that in practice, \ain{} outperforms all algorithms sharing same convergence guarantees: Cubic Newton \citep{nesterov2006cubic} and Globally Regularized Newton \citep{mishchenko2021regularized} and \citet{doikov2021optimization}, and fixed stepsize Damped Newton method (\Cref{sec:experiments}).
	
    \item \textbf{Geometric properties:} We analyze \ain{} under more geometrically natural assumptions. Instead of smoothness, we use a version of self-concordance (\Cref{sec:affine_invariance}), which is invariant to affine transformations and hence also to a choice of a basis. 
    \ain{} preserves affine-invariance obtained from assumptions throughout the convergence.
    In contrast, Cubic Newton uses base-dependent $l_2$ norm and hence depends on a choice of a basis. This represents an extra layer of complexity.

	\item \textbf{Alternative analysis:} We also provide alternative analysis under weaker assumptions
(\Cref{sec:ap_global_nsc}).
\end{itemize} 

The	rest of the paper is structured as follows. In \Cref{ssec:notation} we introduce our notation. In \Cref{sec:alg}, we discuss algorithm \ain{}, affine-invariant properties and self-concordance. In Sections~\ref{sec:global} and \ref{sec:local} we show global  and local convergence guarantees, respectively.
	In \Cref{sec:experiments} we present an empirical comparison of \ain{} with other algorithms sharing fast global convergence.
	
	\vspace{-0.2cm}
	\subsection{Minimization problem \& notation} \label{ssec:notation}
	\begin{mdframed}[backgroundcolor=lightgray!50]
		In the paper, we consider a $d$-dimensional Euclidean space $\mathbb{E}$. Its dual space, $\mathbb{E}^\ast$, is composed of all linear functionals on $\mathbb{E}$. For a functional $g \in \mathbb{E}^\ast$, we denote by $\la g,x\ra$ its value at $x\in \mathbb{E}$.
	\end{mdframed}
	
	We consider the following convex optimization problem:
	\begin{equation}
	\min\limits_{x\in \mathbb{E}}  f(x),
	\label{eq_pr}
	\end{equation} 
	where $f(x) \in C^2$ is a convex function with continuous first and second derivatives and positive definite Hessian. We assume that the problem has a unique minimizer $x_{\ast}\in \argmin\limits_{x\in \mathbb{E}}  f(x)$. 
	Note, that 
	$\nabla f(x) \in \bbE^{\ast}, \quad \nabla^2 f(x) h \in \bbE^{\ast}.$
	Now, we introduce different norms for spaces $\bbE$ and $\bbE^{\ast}$. Denote $x,h \in \bbE, g \in \bbE^{\ast}$. For a self-adjoint positive-definite operator $\mathbf H : \bbE \rightarrow \bbE^\ast$, we can endow these spaces with conjugate Euclidean norms:
	\[
	\|x\|_ {\mathbf H} \eqdef  \la \mathbf Hx,x\ra^{1/2}, \, x \in \bbE, \qquad \|g\|_{\mathbf H}^{\ast}\eqdef\la g,\mathbf H^{-1}g\ra^{1/2},  \, g \in \bbE^{\ast}.
	\]
	For identity $\mathbf H= \mI$, we get classical Eucledian norm $\|x\|_\mI = \la x,x\ra^{1/2}$. For local Hessian norm $\mathbf H = \nabla^2 f(x)$, we use shortened notation
	\begin{equation}
	\label{eq:hessian_norm}
	\gboxeq{\|h\|_x \eqdef \la \nabla^2 f(x) h,h\ra^{1/2}, \, h \in \bbE,} \qquad \gboxeq{\|g\|_{x}^{\ast} \eqdef \la g,\nabla^2 f(x)^{-1}g\ra^{1/2},  \, g \in \bbE^{\ast}.}
	\end{equation}
	Operator norm is defined by 
	\begin{equation}
	\label{eq:matrix_operator_norm}
	    \gboxeq{\normM{\mathbf H} {op} \eqdef \sup_{v\in \mathbb E} \tfrac {\normMd {\mathbf H v} x}{\normM v x}},
	\end{equation} 
	for $\mathbf H :\bbE \rightarrow \bbE^\ast$ and a fixed $x \in \bbE$.
	If we consider a specific case $\bbE \leftarrow \R^d$, then $\mathbf H$ is a symmetric positive definite matrix.
	
	\section{New Algorithm: Affine-Invariant Cubic Newton} \label{sec:alg}
	\vspace{-0.2cm}
    Finally, we are ready to present algorithm \ain{}. It is damped Newton method with updates
	\begin{equation} \label{eq:update}
	x_{k+1} = x_k - \alpha_k \h(x_k)^{-1} \g(x_k),
	\end{equation}
	with stepsize $$\gboxeq{\textstyle \alpha_k \eqdef \frac {-1+\sqrt{1+2 \Lalg \normMd {\g(x_k)} {x_k}}}{\Lalg \normMd {\g(x_k)} {x_k}},}$$ as summarized in \Cref{alg:ain}. 
	Stepsize satisfy $\alpha_k \leq 1$ (from AG inequality, \eqref{eq:AG}). Also $\lim_{x_k \rightarrow x_*} \alpha_k = 1$, hence \eqref{eq:update} converges to Newton method.
	Next, we are going to discuss geometric properties of our algorithm.
	\begin{algorithm} 
		\caption{\ain{}: Affine-Invariant Cubic Newton} \label{alg:ain}
		\begin{algorithmic}[1]
			\State \textbf{Requires:} Initial point $x_0 \in \bbE$, constant $\Lalg$ s.t. $\Lalg \geq \Lsemi>0$
			\For {$k=0,1,2\dots$}
			\State $\alpha_k = \tfrac {-1 +\sqrt{1+2 \Lalg \|\g(x_k)\|_{x_k}^* }}{\Lalg \| \g(x_k)\|_{x_k}^* }$
			\State $x_{k+1} = x_k - \alpha_k \left[\h(x_k)\right]^{-1} \g(x_k)$ \Comment {Note that  $x_{k+1} \stackrel{\eqref{eq:step}}= S_{f,\Lalg}(x_k)$.}
			\EndFor
		\end{algorithmic}
	\end{algorithm}
	\vspace{-0.2cm}
	\subsection{Geometric properties: affine invariance} \label{sec:affine_invariance}
	\vspace{-0.2cm}
	One of the main geometric properties of the Newton method is \textit{affine invariance}, invariance to affine transformations of variables.
	Let $\mathbf A: \bbE \rightarrow \bbE^\ast$ be a non-degenerate linear transformation. Consider function $\phi(y) = f(\mathbf Ay)$. By affine transformation, we denote $f(x) \rightarrow \phi(y) = f(\mathbf Ay), x \rightarrow \mathbf A^{-1}y$. 
% 	\\[2pt]
\vspace{-0.2cm}
	\paragraph{Significance of norms:}
	Note that local Hessian norm $\|h\|_{\nabla f(x)}$ is affine-invariant because
	\[\|z\|_{\nabla^2 \phi(y)}^2=\la\nabla^2 \phi(y)z,z \ra = \la \mathbf A^\top \nabla^2 f(\mathbf A y) \mathbf Az,z\ra = \la \nabla^2 f(x)h,h\ra=\|h\|_{\nabla^2 f(x)}^2,  \]
	where $h=\mathbf Az$. On the other hand, induced norm $\|h\|_\mI$ is not affine-invariant because
	\[\|z\|_\mI^2=\la z,z \ra = \la \mathbf A^{-1}h, \mathbf A^{-1}h\ra = \| \mathbf A^{-1} h\|^2_\mI.  \]
	With respect to geometry, the most natural norm is local Hessian norm, $\|h\|_{\nabla f(x)}$. From affine invariance follows that for this norm, the level sets $\lb y\in \bbE\,|\,\| y-x \|_x^2 \leq c\rb$ are balls centered around $x$ (all directions have the same scaling). In comparison, scaling of the $l_2$ norm is dependent on eigenvalues of the Hessian. In terms of convergence, one direction in $l_2$ can significantly dominate others and slow down an algorithm.
% 	\\[2pt]
	\vspace{-0.2cm}
	\paragraph{Significance for algorithms:}
	Algorithms that are not affine-invariant can suffer from chosen coordinate system. This is the case for Cubic Newton, as its model \eqref{eq:L2-smooth} is bound to base-dependent $l_2$ norm. Same is true for any other method regularized with an induced norm $\|h\|_\mI$.
    % \\[2pt]
    On the other hand, (damped) Newton methods have affine-invariant models, and hence as algorithms independent of the chosen coordinate system. We prove this claim in following lemma (note: $\alpha_k=1$ and $\alpha_k$ from \eqref{eq:update} are affine-invariant).
% 	\\[2pt]
	\begin{lemma}[Lemma 5.1.1 \citet{nesterov2018lectures}] \label{le:AI_newton_nesterov}
		Let the sequence $\lb x_k \rb$ be generated by a damped Newton method with affine-invariant stepsize $\alpha_k$, applied to the function $f$:
		$x_{k+1}=x_k - \alpha_k\left[\nabla^2 f(x_k) \right]^{-1}\nabla f(x_k).$
		For function $\phi(y)$, damped Newton method generates  $\lb y_k \rb$:
		$y_{k+1}=y_k - \alpha_k\left[\nabla^2 \phi(y_k) \right]^{-1}\nabla \phi(y_k),$
		with $y_0 = \mathbf A^{-1} x_0$. Then $y_k= \mathbf A^{-1} x_k$.
	\end{lemma}
	\vspace{-0.2cm}
	\subsection{Significance in assumptions: self-concordance}
	\vspace{-0.2cm}
	We showed that damped Newton methods preserve affine-invariance through iterations. Hence it is more fitting to analyze them under affine-invariant assumptions. Affine-invariant version of smoothness, \emph{self-concordance}, was introduced in \citet{nesterov1994interior}.
	
	\begin{mdframed} [backgroundcolor=lightgray!50]
		\begin{definition} \label{def:self-concordance}
			Convex function $f \in C^3$ is called \textit{self-concordant}  if
			\begin{equation}
			\label{eq:self-concordance}
			|D^3 f(x)[h]^3| \leq \Lstandard\|h\|_{x}^3, \quad \forall x,h\in \bbE,
			\end{equation}
			where for any integer $p\geq 1$, by $D^p f(x)[h]^p \eqdef D^p f(x)[h,\ldots,h]$ we denote the $p$-th order directional derivative\footnote{For example, $D^1 f(x)[h] =\langle\nabla  f(x),h\rangle$ and $D^2 f(x)[h]^2 = \la \nabla^2 f(x) h, h \ra $.} of $f$ at $x\in \bbE$ along direction $h\in \bbE$.
		\end{definition}
	\end{mdframed}
	Both sides of inequality are affine-invariant. This assumption corresponds to a big class of optimization methods called interior-point methods. Self-concordance implies uniqueness of the solution, as stated in following proposition.
	\begin{proposition}[Theorem 5.1.16, \citet{nesterov2018lectures}] \label{prop:unique_solution}
		Let a self-concordant function $f$ be bounded below. Then it attains its minimum at a single point.
	\end{proposition}
	
	\citet{rodomanov2021greedy} introduced stronger version of self-concordance assumption.
	\begin{mdframed} [backgroundcolor=lightgray!50]
		\begin{definition}	\label{def:strong-self-concordance}
			Convex function $f \in C^2$  is called \textit{strongly self-concordant} if
			\begin{equation}
			\label{eq:strong-self-concordance}
			\nabla^2 f(y)-\nabla^2 f(x) \preceq \Lstrongly \|y-x\|_{z} \nabla^2 f(w), \quad \forall y,x,z,w \in \bbE .
			\end{equation} 
		\end{definition}
	\end{mdframed}
	For our analysis, we introduce definition for functions between self-concordant and strongly self-concordant.
	
	\begin{mdframed}[backgroundcolor=lightgray!50]
		\begin{definition} \label{def:semi-self-concordance}
			Convex function $f \in C^2$  is called \textit{semi-strongly self-concordant} if
			\begin{equation}
			\label{eq:semi-strong-self-concordance}
			\left\|\nabla^2 f(y)-\nabla^2 f(x)\right\|_{op} \leq \Lsemi \|y-x\|_{x} , \quad \forall y,x \in \bbE .
			\end{equation} 
		\end{definition}
	\end{mdframed}
	
	Note that all of the Definitions \ref{def:self-concordance} - \ref{def:semi-self-concordance} are affine-invariant, their respective classes satisfy
	\begin{gather*}
	\textit{strong self-concordance} \subseteq \textit{semi-strong self-concordance}
	\subseteq \textit{self-concordance}.
	\end{gather*}
	Also, for a fixed strongly self-concordant function $f$ and smallest such  $\Lstandard, 
	\Lsemi, \Lstrongly$ holds
	$ \Lstandard
	\leq \Lsemi \leq \Lstrongly. $

	These notions are related to the convexity and smoothness; strong concordance follows from function $L_2$-Lipschitz continuous Hessian and strong convexity.
	\begin{proposition}[Example 4.1 from \citep{rodomanov2021greedy}] \label{le:sscf}
		Let $\mathbf H: \mathbb{E} \rightarrow \mathbb{E}^{*}$ be a self-adjoint positive definite operator. Suppose there exist $\mu>0$ and $L_{2} \geq 0$ such that the function $f$ is $\mu$-strongly convex and its Hessian is $L_{2}$-Lipschitz continuous \eqref{eq:L2-smooth} with respect to the norm $\|\cdot\|_{\mathbf H}$. Then $f$ is strongly self-concordant with constant $\Lstrongly=\tfrac{L_{2}}{\mu^{3 / 2}}$. 
	\end{proposition}

	\vspace{-0.2cm}
    \subsection{From assumptions to algorithm}
	\vspace{-0.2cm}
	From semi-strong self-concordance we can get a second-order bounds on the function and model.
	
	\begin{lemma} \label{le:model}
		If $f$ is semi-strongly self-concordant, then
		\begin{equation}
		\label{eq:lower_upper_bound}
		\left|f(y) - Q_f(y;x) \right|  \leq \tfrac{\Lsemi}{6}\|y-x\|_x^3, \quad \forall x,y\in \bbE.
		\end{equation}
		Consequently, we have upper bound for function value in form
		\begin{equation} \label{eq:wssc_ub}
		f(y) \leq Q_{f}(y;x) +\tfrac{\Lsemi}{6}\|y-x\|_{x}^3.
		\end{equation}
	\end{lemma}
	One can show that \eqref{eq:wssc_ub} is not valid for just self-concordant functions. For example, there is no such upper bound for $-\log(x)$. Hence, the semi-strongly self-concordance is significant as an assumption.
	
	We can define iterates of optimization algorithm to be minimizers of the right hand side of \eqref{eq:wssc_ub}, 
	\begin{mdframed} [backgroundcolor=lightgray!50]
    	\begin{gather} \label{eq:step}
    	S_{f,\Lalg}(x)\eqdef x + \argmin_{h\in \bbE} \left\lbrace f(x) + \la \nabla f(x),h\ra + \tfrac{1}{2} \la \nabla^2 f(x)h,h\ra +\tfrac{\Lalg}{6}\|h\|_{x}^3\right\rbrace,\\
    	\label{eq:S_step}
    	x_{k+1} = S_{f,\Lalg}(x_{k}),
    	\end{gather} 
	\end{mdframed}
	for an estimate constant $\Lalg \geq \Lsemi$. 
	It turns out that subproblem \eqref{eq:step} is easy to solve. To get an explicit solution, we compute its gradient w.r.t. $h$. For solution $h^{\ast}$, it should be equal to zero, 
	\begin{gather}
	 \nabla f(x) + \nabla^2 f(x)h^{\ast} +\tfrac{\Lalg}{2}\|h^{\ast}\|_{x}\nabla^2 f(x) h^{\ast} = 0
	 \label{eq:h_star_main},\\
	  h^{\ast} = -\left[\nabla^2 f(x)\right]^{-1}\nabla f(x)\cdot \left(\tfrac{\Lalg}{2}\|h^{\ast}\|_{x}+ 1\right)^{-1}. \label{eq:h_star2}
	\end{gather}
	We get that step \eqref{eq:S_step} has the same direction as a Newton method and is scaled by  $\alpha_k =\left(\tfrac{\Lalg}{2}\|h^{\ast}\|_{x}+ 1\right)^{-1}. $ 
	Now, we substitute $h^{\ast}$ from \eqref{eq:h_star2} to \eqref{eq:h_star_main}
	\begin{gather}
	 \nabla f(x) - \nabla f(x) \alpha_k  -\tfrac{\Lalg}{2}\left<\nabla f(x),\left[\nabla^2 f(x)\right]^{-1}\nabla f(x) \right>^{1/2} \nabla f(x) \alpha_k^2  = 0, \notag\\
	 \nabla f(x) \left(1- \alpha_k  -\tfrac{\Lalg}{2}\|\nabla f(x)\|_x^{\ast} \alpha_k^2\right)  = 0. \label{eq:h_quad_eq}
	\end{gather}
	We solve the quadratic equation \eqref{eq:h_quad_eq} for $\alpha_k$, and obtain explicit formula for stepsizes of \ain{}, as \eqref{eq:update}.
	We formalize this connection in theorem, for further explanation see proof in \Cref{sec:ap_theory}.
	\begin{mdframed}
		\begin{theorem} \label{th:stepsize}
			For $\Lalg \geq \Lsemi$, update of \ain{} \eqref{eq:update},
			\[x_{k+1} = x_k - \alpha_k \h(x_k)^{-1} \g(x_k), \qquad \text{where} \qquad \alpha_k = \tfrac {-1+\sqrt{1+2 \Lalg \normMd {\g(x_k)} {x_k}}}{\Lalg \normMd {\g(x_k)} {x_k}},\]
			is a minimizer of upper bound \eqref{eq:step}, $x_{k+1} = S_{f,\Lalg}(x_k)$.
		\end{theorem}
	\end{mdframed}
	
\section{Convergence Results}	
	\vspace{-0.2cm}
	\subsection{Global convergence} \label{sec:global}
	\vspace{-0.2cm}
	Next, we focus on global convergence guarantees.
	We will utilize the following assumption:
		\begin{mdframed}
		\begin{assumption}[Bounded level sets] \label{as:level_sets}
			The objective function $f$ has a unique minimizer $x_{*}$.
			Also, the diameter of the level set $\mathcal{L}(x_0)\eqdef\lb x \in \bbE : f(x)\leq f(x_0) \rb$ is bounded by a constant $D_2$ as\footnote{We state it in $l_2$ norm for easier verification. In proofs, we use its variant $D$ in Hessian norms, \eqref{D_lebeg}.}, $\max\limits_{x\in\mathcal{L}(x_0)} \normM {x-x_{*}} 2 \leq D_2<+\infty$.
		\end{assumption}
	\end{mdframed}
	
	Our analysis proceeds as follows. Firstly, we show that one step of the algorithm decreases function value, and secondly, we use the technique from  \citep{ghadimi2017second} to show that multiple steps lead to $\cO\left ({k^{-2}} \right )$ global convergence.
	We start with following lemma.
	\begin{lemma}[One step globally] \label{lm:main_lemma}
		Let function $f$ be $\Lsemi$-semi-strongly self-concordant, convex with positive-definite Hessian and $\Lalg \geq \Lsemi$. Then for any $x \in  \bbE$,  we have
		\begin{equation} \label{eq:min_lemma}
		S_{f,\Lalg}(x) \leq \min \limits_{y \in \bbE} \left\{ f(y) + \tfrac{\Lalg}{ 3}\|y -x\|^3_x \right\}.
		\end{equation}
	\end{lemma}
	This lemma implies that step \eqref{eq:step} decreases function value (take $y \leftarrow x$). Using notation of \Cref{as:level_sets}, $x_k \in \mathcal{L}(x_0) $ for any $k\geq 0 $. Also, setting $y \leftarrow x$ and $x \leftarrow x_{\ast}$ in \eqref{eq:semi-strong-self-concordance} yields
	$\|x-x_{\ast}\|_x \leq \ls \|x-x_{\ast}\|_{x_{\ast}}^2 + \Lalg \|x-x_{\ast}\|_{x_{\ast}}^3 \rs^{\tfrac{1}{2}}.$
	We denote those distances $D$ and $R$,
	\begin{equation} \label{D_lebeg}
	\gboxeq{D \eqdef \max\limits_{t \in [0;k+1]} \|x_t - x_{\ast}\|_{x_t}}  \qquad \text{and} \qquad \gboxeq{R \eqdef \max\limits_{x\in \mathcal{L}(x_0)} \ls \|x-x_{\ast}\|_{x_{\ast}}^2 + \Lalg \|x-x_{\ast}\|_{x_{\ast}}^3 \rs^{\tfrac{1}{2}}.}
	\end{equation}
	
	They are both affine-invariant and 
	$R$ upper bounds $D$. While $R$ depends only on the level set $\mathcal L (x_0)$, $D$ can be used to obtain more tight inequalities. We avoid using common distance $D_2$, as $l_2$ norm would ruin affine-invariant properties.

	\begin{mdframed}
		\begin{theorem} \label{thm:convergence}
			Let $f(x)$ be a $\Lsemi$-semi-strongly self-concordant convex function with positive-definite Hessian, constant $\Lalg$ satisfy $\Lalg \geq \Lsemi$ and \Cref{as:level_sets} holds. Then, after $k+1$ iterations of Algorithm \ref{alg:ain}, we have the following convergence rate: 
			\begin{equation}
			f(x_{k+1}) - f(x_{\ast})\leq O\ls\tfrac{\Lalg D^{3}}{k^{2}}\rs\leq O\ls\tfrac{\Lalg R^{3}}{k^2}\rs.  \label{eq:convergence}
			\end{equation}	
		\end{theorem}
	\end{mdframed}
	Consequently, \ain{} converges globally with a fast rate $\cO \left( k^{-2} \right)$. We can now present local analysis.
	\vspace{-0.2cm}
	\subsection{Local convergence} \label{sec:local}
	\vspace{-0.2cm}
	For local quadratic convergence are going to utilise following lemmas.
	\begin{lemma} 
	\label{le:lsconv}
		For convex $\Lsemi$-semi-strongly self-concordant function $f$ and for any $0<c<1$ in the neighborhood of solution \begin{equation}
		\label{eq_first_neigborhood}
		 x_k \in \left\{ x :  \normMd {\g (x)} {x} \leq \tfrac{(2c+1)^2-1}{2\Lalg} \right\}\,\, \text{holds}\,\,  \h(x_{k+1})^{-1} \preceq \left( 1 - c \right)^{-2} \h(x_k)^{-1}. 
		 \end{equation}
	\end{lemma}
	\Cref{le:lsconv} formalizes that a inverse hessians of a self-concordant function around the solution is non-degenerate.
	With this result, we can show one-step gradient norm decrease.
	
	\begin{lemma}[One step decrease locally] \label{le:one_step_local}
		Let function $f$ be $\Lsemi$-semi-strongly self-concordant and $\Lalg \geq \Lsemi$. If $x_k$ such that \eqref{eq_first_neigborhood} holds, then for next iterate $x_{k+1}$ of \ain{} holds
		\begin{equation} \label{eq:one_step_local}
		\normMd {\g(x_{k+1})} {x_k}
		\leq \Lalg \alpha_k^2 \normsMd{\g(x_k)} {x_k}
		< \Lalg \normsMd{\g(x_k)} {x_k}.
		\end{equation}
		Using \Cref{le:lsconv}, we shift the gradient bound to respective norms,
		\begin{equation} \label{eq:one_step_local_shifted}
		\normMd {\g(x_{k+1})} {x_{k+1}}
		\leq  \tfrac{\Lalg \alpha_k^2}{1-c} \normsMd{\g(x_k)} {x_k}
		< \tfrac{\Lalg \alpha_k^2}{1-c}\normsMd{\g(x_k)} {x_k}.
		\end{equation}
		Gradient norm decreases $\normMd{\g(x_{k+1})} {x_{k+1}} \leq \normMd{\g(x_k)} {x_k}$ for $\normMd{\g(x_k)} {x_k} \leq \tfrac {(2-c)^2-1} {2\Lalg }$.
	\end{lemma}
	As a result, neighbourhood of the local convergence is $\left\{x: \normMd{\g(x)} {x} \leq \min \left[\tfrac {(2-c)^2-1} {2\Lalg };\tfrac{(2c+1)^2-1}{2\Lalg}\right] \right\}.$
	Maximizing by $c$, we get $c=1/3$ and neighrbourhood $\left\{x: \normMd{\g(x)} {x} \leq \tfrac{8}{9\Lalg } \right\}.$ One step of \ain{} decreases gradient norm quadratically, multiple steps leads to following decrease.
	\begin{mdframed}
		\begin{theorem}[Local convergence rate] \label{th:local}
			Let function $f$ be $\Lsemi$-semi-strongly self-concordant, 
			$\Lalg \geq \Lsemi$ and starting point $x_0$ be in the neighborhood of the solution such that $\normMd{\g(x_{0})} {x_{0}} \leq \tfrac {8} {9\Lalg} $.
			For $k\geq 0$, we have quadratic decrease of the gradient norms,
			\begin{align}
			\normMd{\g(x_{k})}{x_{k}}
			&\leq \left( \tfrac {3} {2} \Lalg \right )^k 
			\left( \normMd{\g(x_{0})} {x_{0}} \right)^{2^k} .
			\end{align}
		\end{theorem}
	\end{mdframed}
	\vspace{-0.2cm}
	\section{Numerical Experiments} \label{sec:experiments}
	\vspace{-0.2cm}
	In this section, we evaluate proposed \ain{} (\Cref{alg:ain}) algorithm on the logistic regression task and second-order lower bound function. 
	We compare it with regularized Newton methods sharing fast global convergence guarantees: Cubic Newton method \citep{nesterov2006cubic}, and Globally Regularized Newton method (\citep{mishchenko2021regularized,doikov2021optimization}) with $L_2$-constant. Because \ain{} has a form of a damped Newton method, we also compare it with Damped Newton with fixed (tuned) stepsize $\alpha_k=\alpha$. We report decrease in function value $f(x_k)$, function value suboptimality $f(x_k)-f(x_*)$ with respect to iteration and time. The methods are  implemented as PyTorch optimizers. The code is available at \url{https://github.com/OPTAMI}. 
	
\subsection{Logistic regression} 
\vspace{-0.2cm}
	For first part, we solve the following empirical risk minimization problem:
	\[ \textstyle
	\min \limits_{x \in \mathbb{R}^d} \left\{ f(x) = \frac{1}{m} \sum \limits_{i=1}^m  \log\left(1-e^{-b_i a_i^\top x}\right) + \frac{\mu}{2} \|x\|_2^2 \right\},
	\]
	where $\{(a_i, b_i)\}_{i=1}^m$ are given data samples described by features $a_i$ and class $b_i \in \{-1, 1\}$. 
	%\\[2pt]
	
	In \Cref{fig:a9a}, we consider task of classification images on dataset \emph{a9a} \citep{chang2011libsvm}. Number of features for every data sample is $d=123$, $m=20000$. We take staring point $x_0 \eqdef 10 [1,1, \dots, 1]^\top$ and $\mu = 10^{-3}$. Our choice is differ from $x_0=0$ (equal to all zeroes) to show globalisation properties of methods ($x_0=0$ is very close to the solution, as  Newton method converges in 4 iterations). Parameters of all methods are fine-tuned, we choose parameters $\Lalg,L_2, \alpha$ (of \ain{}, Cubic Newton, Damped Newton, resp.) to largest values having monotone decrease in reported metrics. Fine-tuned values are $\Lalg = 0.97$ $L_2= 0.000215$, $\alpha = 0.285$.
	\Cref{fig:a9a} demonstrates that \ain{} converges slightly faster than Cubic Newton method by iteration, notably faster than Globally Regularized Newton and significantly faster than Damped Newton. \ain{} outperforms every method by time.
   \begin{figure}[t]
	%\centering
	%\vspace{-0.1cm}
	\includegraphics[width=0.32\linewidth]{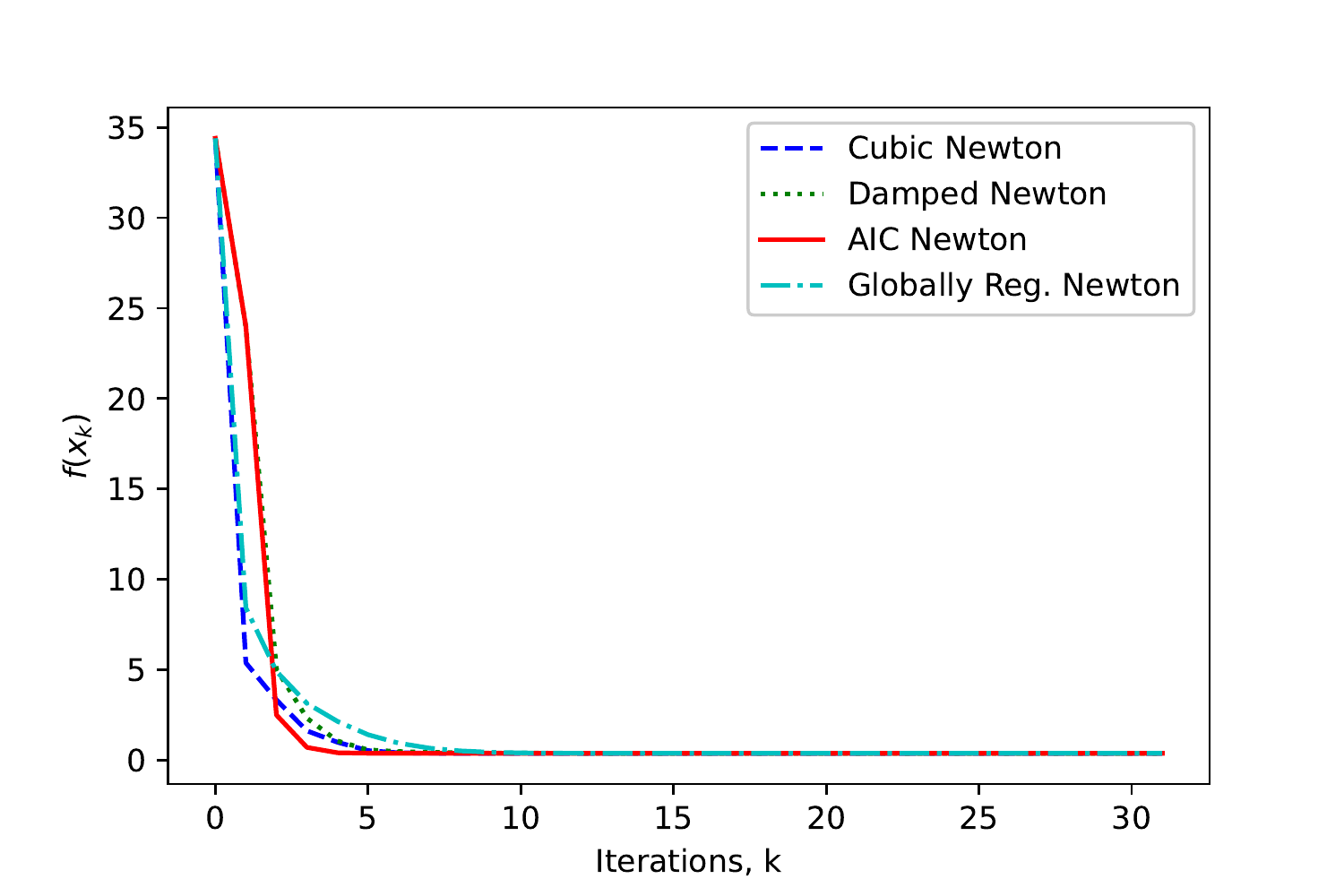}
	\includegraphics[width=0.32\linewidth]{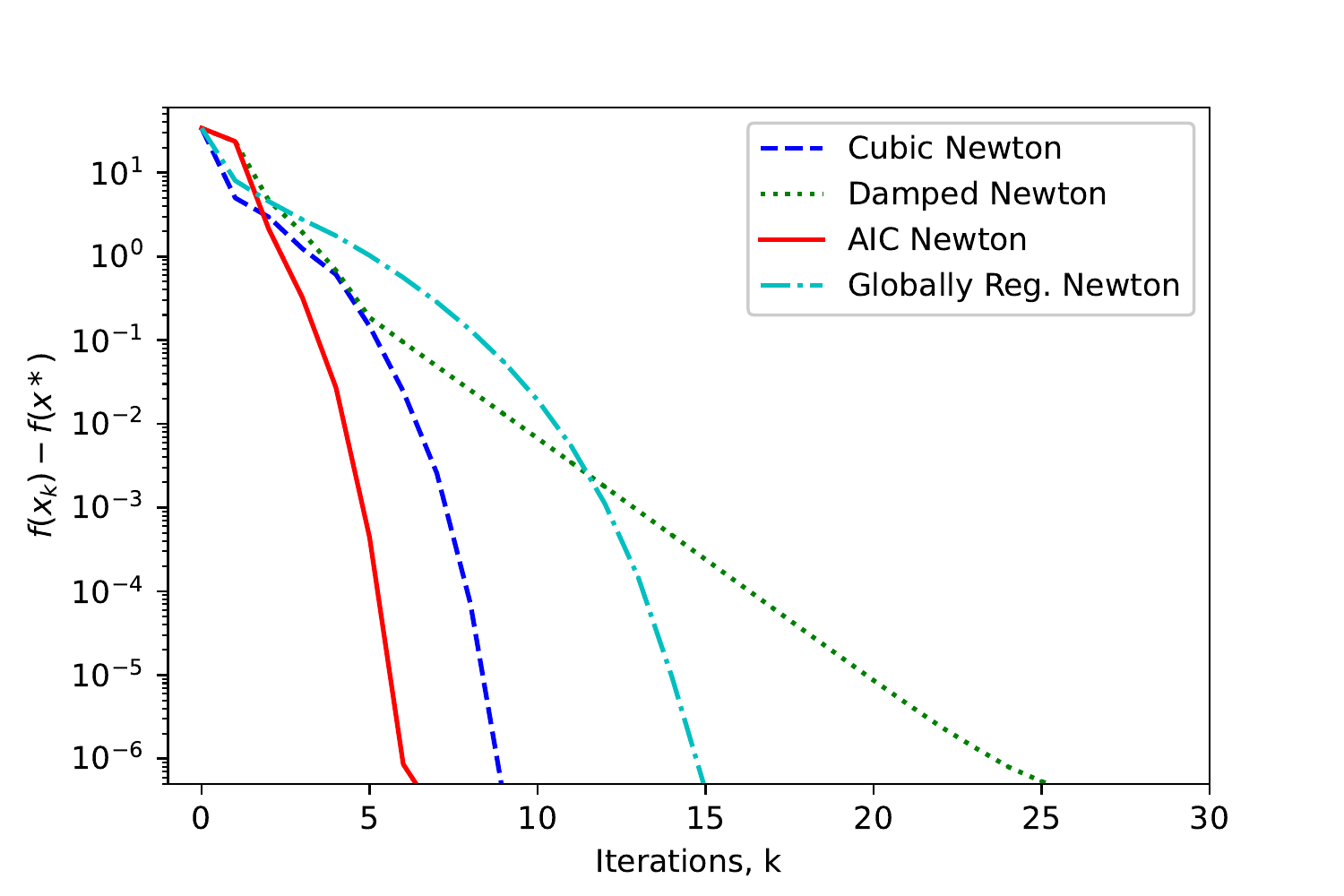}
	\includegraphics[width=0.32\linewidth]{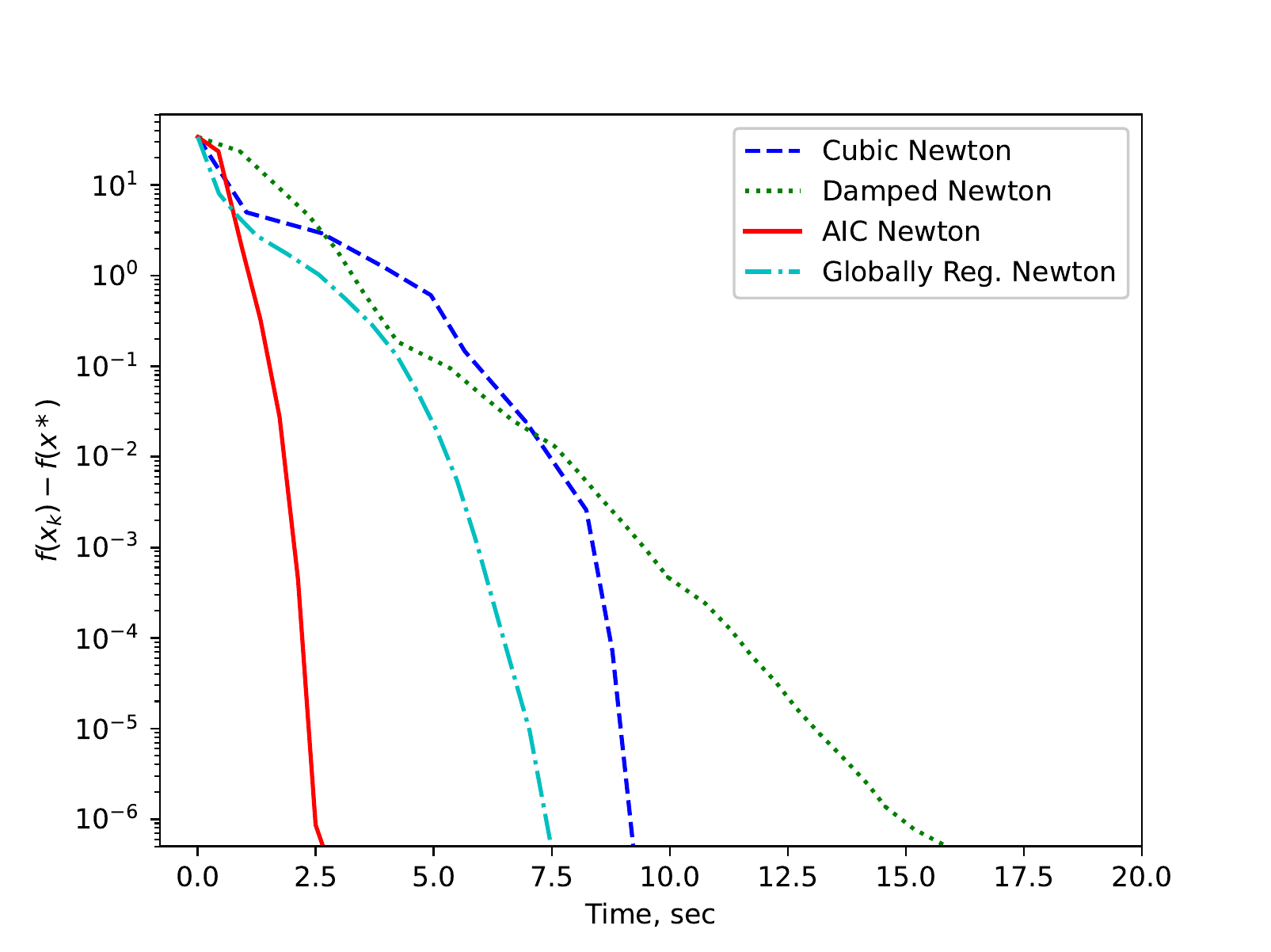}
	\caption{Comparison of regularized Newton methods and Damped Newton method for logistic regression task on  \emph{a9a} dataset.}
	\label{fig:a9a}
	\vspace{-0.5cm}
	\end{figure}
	
	\begin{figure}
	
	    \vspace{-0.35cm}
		%\centering
		\includegraphics[width=0.32\linewidth]{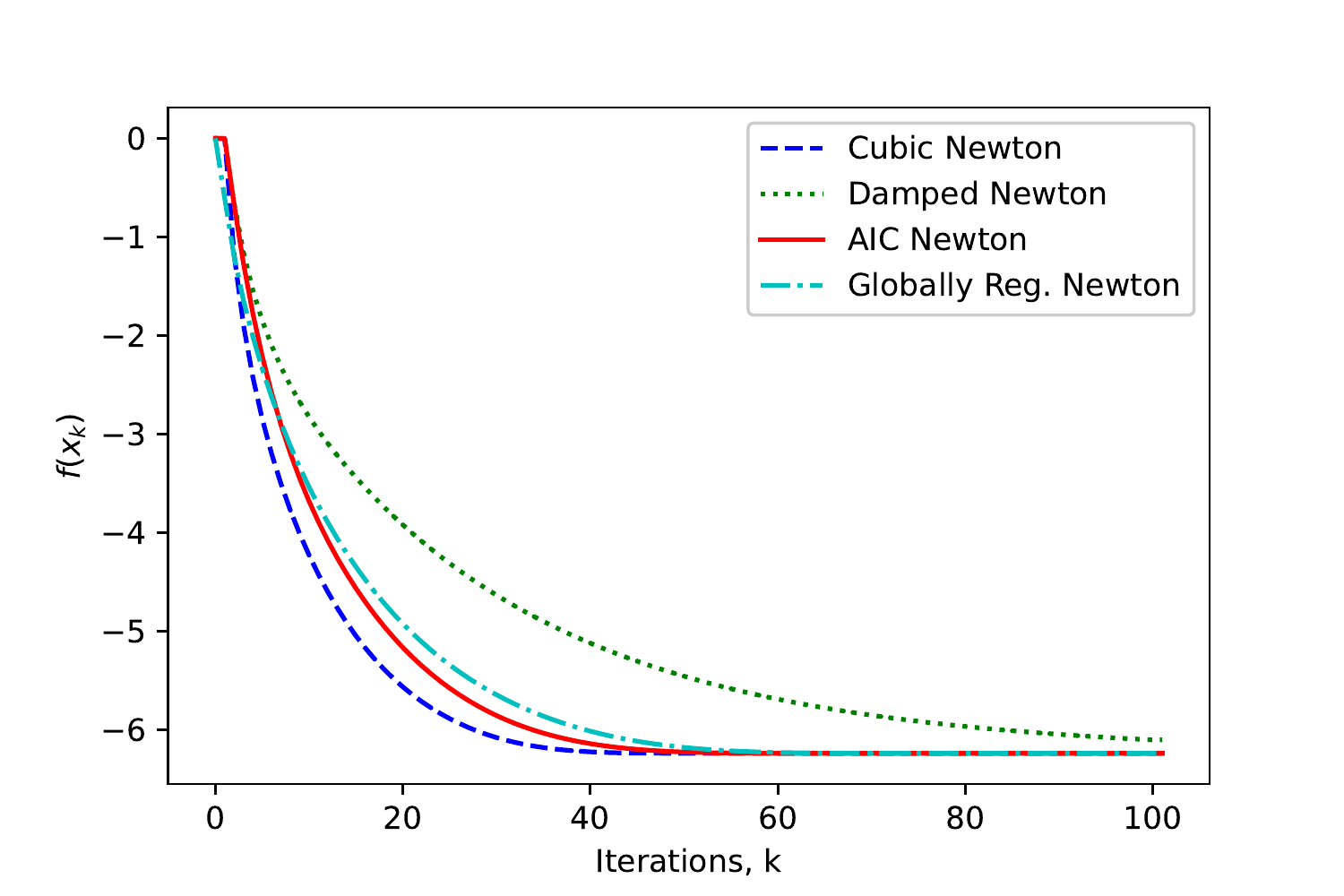}
		\includegraphics[width=0.32\linewidth]{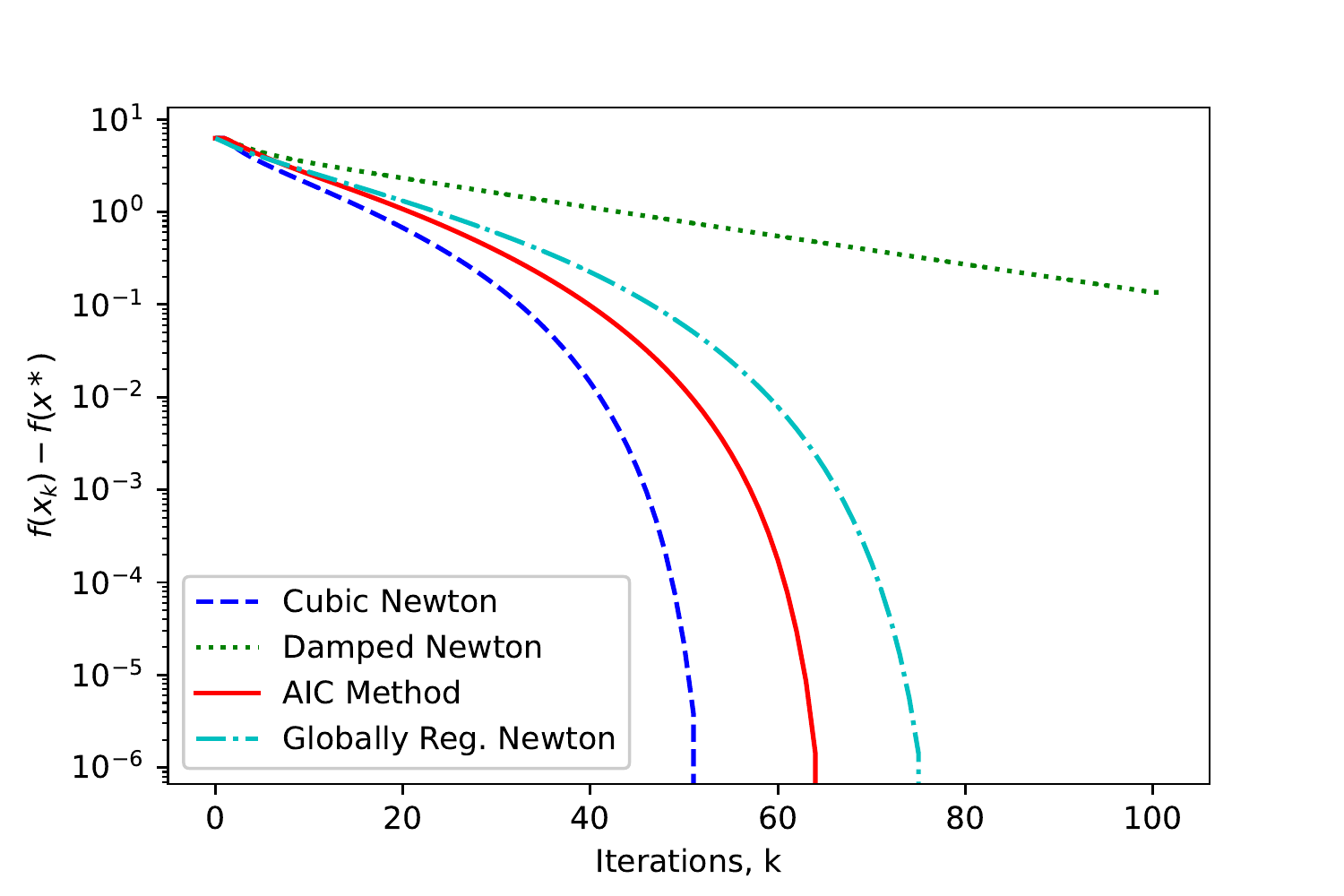}
		\includegraphics[width=0.29\linewidth]{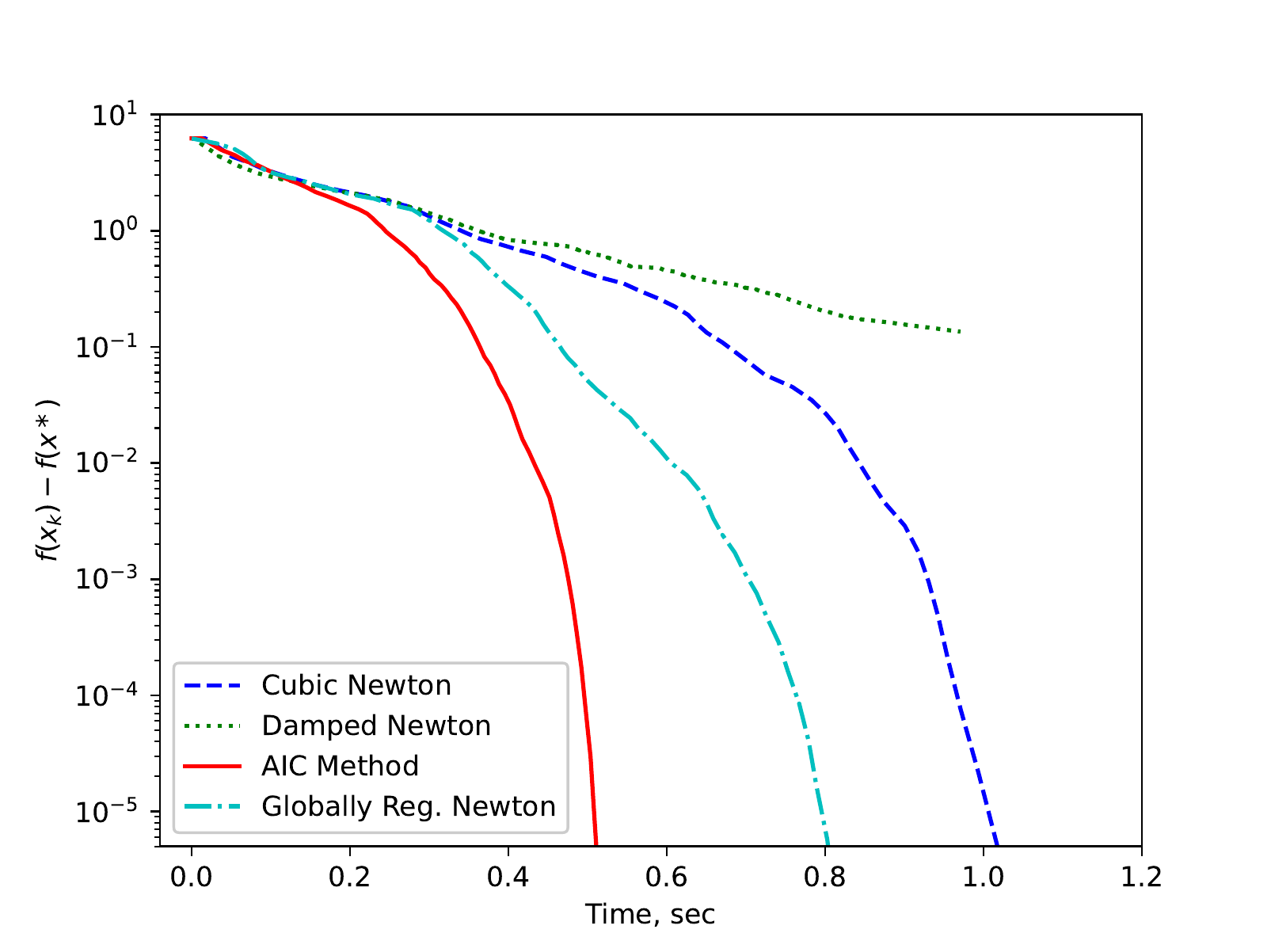}
		\caption{Comparison of regularized Newton methods and Damped Newton method for \emph{second-order lower bound function}.}
		\label{fig:lower_bound}
		\vspace{-0.7cm}
	\end{figure} 	
\vspace{-0.2cm}	
\subsection{Second-order lower bound function}
\vspace{-0.2cm}
	For second part we solve the following minimization problem:
	\[
	\min_{x \in \mathbb{R}^d} \left\{ f(x) = \tfrac{1}{d} \sum_{j=1}^d  \left| \left[\mathbf A x\right]_j\right|^3 - x_1 + \tfrac{\mu}{2} \|x\|_2^2 \right\}, \quad \text{where } \quad 
	\mathbf A = \begin{pmatrix}
        1 & -1 & 0  &\dots & 0 \\
        0 & 1 & -1  &\dots & 0\\ 
        & & \dots & \dots &  \\
        %0 & 0  &\dots & 1 & -1\\
        0 & 0  &\dots & 0  & 1
        \end{pmatrix}.
	\]
	This function is a lower bound for a class of functions with Lipschitz continuous Hessian \eqref{eq:L2-smooth} with additional regularization \citep{nesterov2006cubic, nesterov2021implementable}. 
	In \Cref{fig:lower_bound}, we take $d=20$, $x_0 = 0$ (equal to all zeroes). Parameters $\Lalg,L_2, \alpha$ are fine-tuned to largest values having monotone decrease in reported metrics: $\Lalg = 662$ $L_2= 0.662$, $\alpha = 0.0172$. 
    \Cref{fig:lower_bound} demonstrates that \ain{} converges slightly slower than Cubic Newton method, slightly faster than Globally Regularized Newton, and significantly faster than Damped Newton. ain{} outperforms every method by time. More experiments are presented in \Cref{sec:ap_experiments}. Note, that the iteration of the Cubic Newton method needs an additional line-search, so one iteration of Cubic Newton is computationally harder than one iteration of \ain{}. More experiments are presented in \Cref{sec:ap_experiments}.
	\vspace{-0.2cm}
	\subsection*{Acknowledgement}
	\vspace{-0.2cm}
	The work of D. Pasechnyuk and A. Gasnikov was supported by a grant for research centers in the field of artificial intelligence, provided by the Analytical Center for the Government of the RF in accordance with the subsidy agreement (agreement identifier 000000D730321P5Q0002) and the agreement with the Ivannikov Institute for System Programming of the RAS dated November 2, 2021 No. 70-2021-00142.

	\bibliographystyle{plainnat}
	\bibliography{AICNewton}

\clearpage	
\part*{Appendix}
	\appendix

	\section{Extra comparisons to other methods} \label{sec:ap_experiments}
	
	\subsection{Damped Newton method stepsize comparison}
	In \Cref{fig:stepsizes}, we present comparison of stepizes of \ain{} with other damped Newton methods \citep{nesterov2018lectures}. Our algorithm uses stepsize bigger by orders of magnitude. For reader's convenience, we repeat stepsize choices. For \ain{} stepsize is $\alpha = \frac {-1 + \sqrt{1+ 2G}}{G}$, where $G \eqdef \Lsemi \normMd {\g(x_k)} {x_k}$. For damped Newton methods from \citet{nesterov2018lectures}, $\alpha_1 \eqdef \frac 1 {1 + G_1}, \alpha_2 \eqdef \frac {1+ G_1} {1 + G_1 + G_1^2}$, where $G_1 \eqdef \Lstandard \normMd {\g(x_k)} {x_k}$. 
	
	\begin{figure}
		\centering
		\includegraphics[width=0.49\textwidth]{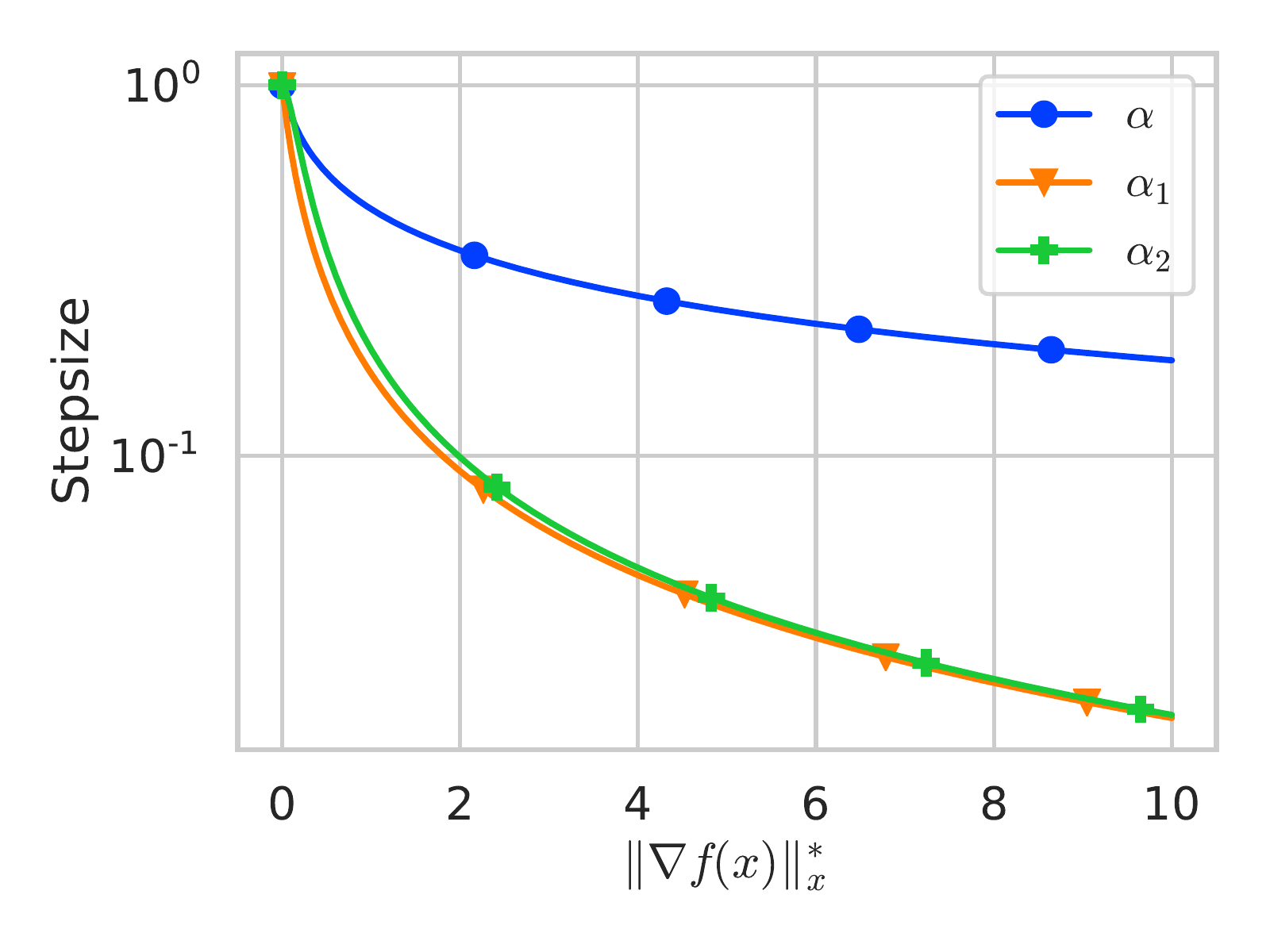}
		\includegraphics[width=0.49\textwidth]{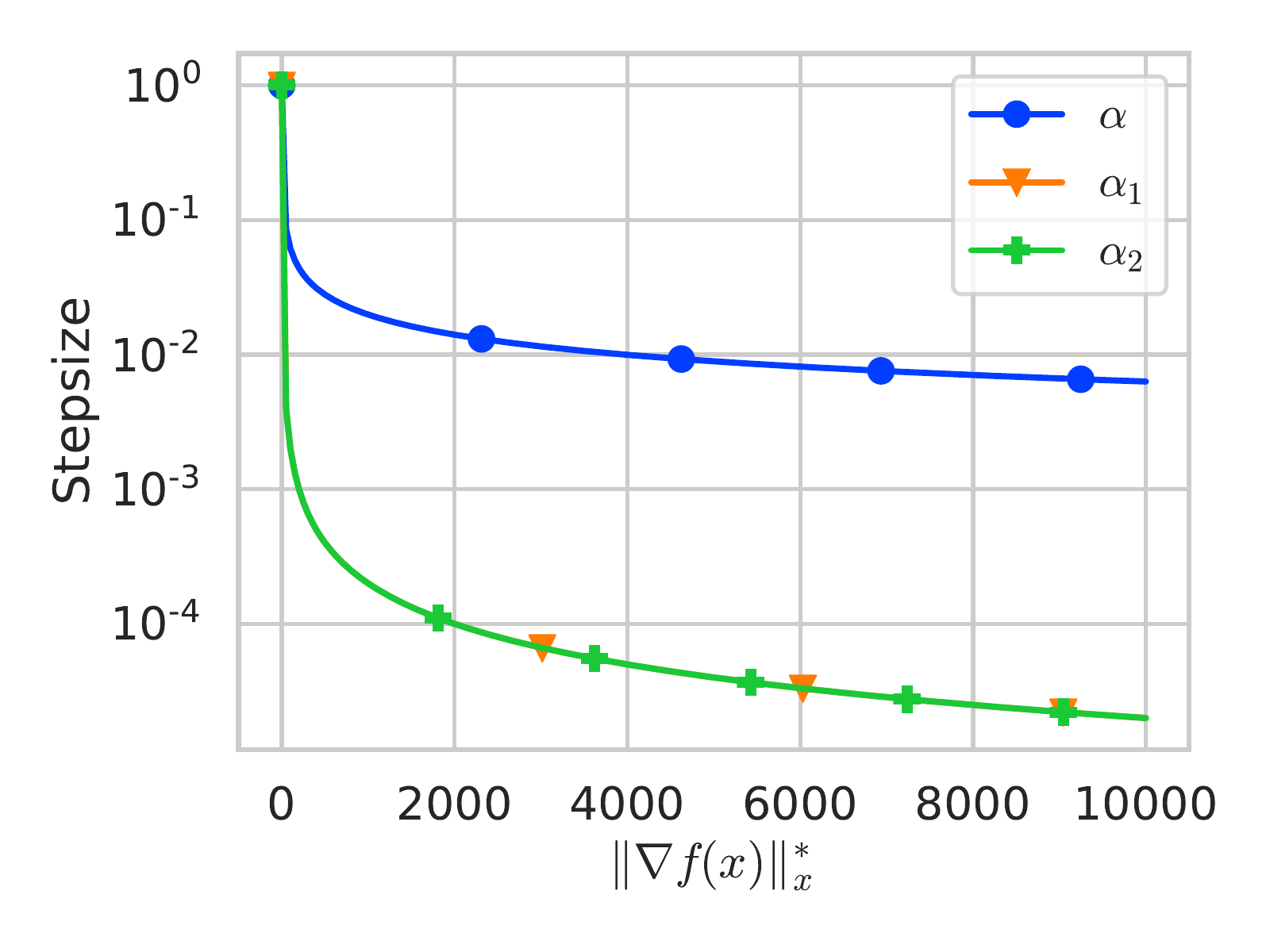}
		\caption{Comparison of stepsizes of affine-invariant damped Newton methods with quadratic local convergence. We compare \ain{} (blue), and stepsizes from \citet{nesterov2018lectures} in orange and green. We set $\Lsemi =\Lstandard=5.$}
		\label{fig:stepsizes}
	\end{figure}
	
	\subsection{Convergence rate comparison under various assumptions}
	In this subsection we present \Cref{tab:rates} -- comparison of \ain{} to regularized Newton methods under different smoothness and convexity assumptions. 
	
	\begin{table*}[!t]
		\centering
		\setlength\tabcolsep{0pt} % default value: 6pt
		\begin{threeparttable}[b]
			{\scriptsize
				\renewcommand\arraystretch{2.2}
				\caption{
				Convergence guarantees under different versions of convexity and smoothness assumptions. For simplicity, we disregard dependence on bounded level set assumptions. All compared assumptions are considered for $\forall x,h \in \R^d$. We highlight the best know rates in blue.}
				\label{tab:rates}
				\centering 
				\begin{tabular}{ccccccc}\toprule[.1em]
					\bf Algorithm & \bf  \makecell{Strong \\ convexity \\ constant}  &\bf  \makecell{Smoothness assumption} &\bf  \makecell{Global \\ convergence \\ rate} & \bf \makecell{Local \tnote{\color{red}(1)} \\convergence\\exponent} & \bf Reference \\
					\midrule
					\makecell{Damped Newton B} &$0$\tnote{\color{red}(2)} & self-concordance (\Cref{def:self-concordance}) & $\cO\left( k^{-\frac 12} \right)$ & {\color{blue} $ 2$} &  \citet[$(5.1.28)$]{nesterov2018lectures} \\
					\makecell{Damped Newton C} &$0$\tnote{\color{red}(2)} & self-concordance (\Cref{def:self-concordance}) & \xmark & {\color{blue} $ 2$} &  \citet[$(5.2.1)_C$]{nesterov2018lectures} \\
				    \makecell{Cubic Newton } & $\mu$ & Lipschitz-continuous Hessian \eqref{eq:L2-smooth} & \makecell{$\cO\left( {\color{blue}{ k^{-2}} } \right) $} & \makecell{{\color{blue} $2$}} &   \makecell{\citet{nesterov2006cubic}\\ \citet{doikov2021local}} \\
				    \makecell{Cubic Newton } & $\mu$-star-convex & Lipschitz-continuous Hessian \eqref{eq:L2-smooth} & \makecell{$\cO\left( {\color{blue}{ k^{-2}} } \right) $} & $\frac 32$ &   \citet{nesterov2006cubic} \\
				    \makecell{Cubic Newton } & \makecell{non-convex, \\ bounded below} & Lipschitz-continuous Hessian \eqref{eq:L2-smooth} & \makecell{$\cO\left( {\color{blue}{ k^{-\frac 23}} } \right) $} & \xmark &   \citet{nesterov2006cubic} \\
				    \makecell{Globally Reg. Newton} & $\mu$ & Lipschitz-continuous Hessian \eqref{eq:L2-smooth}
				    & \makecell{$\cO\left( {\color{blue}{ k^{-2}} } \right) $} & \makecell{$\frac 32$} &   \makecell{\citet{mishchenko2021regularized} \\ \citet{doikov2021optimization}} \\
				    \makecell{Globally Reg. Newton} & $0$ & Lipschitz-continuous Hessian \eqref{eq:L2-smooth}
				    & \makecell{$\cO\left( {\color{blue}{ k^{-2}} } \right) $} & \xmark &   \makecell{\citet{mishchenko2021regularized} \\ \citet{doikov2021optimization}} \\
					\midrule
					\makecell{\textbf{AIC Newton}} & $0$\tnote{\color{red}(2)} & semi-strong self-concordance (\Cref{def:semi-self-concordance}) & \makecell{$\cO\left( {\color{blue}{ k^{-2}} } \right) $} & \makecell{{\color{blue} $2$}} &   Theorems \ref{thm:convergence}, \ref{th:local} \\
					\makecell{\textbf{AIC Newton}} & $\mu$ & $f(x+h) - f(x) \leq \ip{\g(x)} {h} + \tfrac 1 2 \normM h x ^2 + \tfrac \Lalt 6 \normM h x ^3$ & \makecell{$\cO\left( {\color{blue}{ k^{-2}} } \right) $} & \makecell{{\color{blue} $2$}} & Theorems \ref{th:global}, \ref{th:local} \tnote{\color{red}(3)} \\
					\makecell{\textbf{AIC Newton}} & $0$ & $f(x+h) - f(x) \leq \ip{\g(x)} {h} + \tfrac 1 2 \normM h x ^2 + \tfrac \Lalt 6 \normM h x ^3$ & \makecell{$\cO\left( {\color{blue}{ k^{-2}} } \right) $ \tnote{\color{red}(4)}} & \xmark & \Cref{th:global} \\
					\bottomrule[.1em]
				\end{tabular}
			}
			\begin{tablenotes}
				{\scriptsize
					\item [\color{red}(1)] For a Lyapunov function $\Phi$ and a constant $c>0$, we report exponent $\beta>1$ of $\Phi(x_{k+1}) \leq c \Phi (x_k)^\beta$. Mark \xmark{} means that such $\beta, c, \Phi$ are not known.
					\item [\color{red}(2)] Self-concordance implies strong convexity locally.
					\item [\color{red}(3)] Under strong convexity, we can prove local convergence analogically to \Cref{th:local}.
					\item [\color{red}(4)] Convergence to a neighborhood of the solution
				}
			\end{tablenotes}
		\end{threeparttable}
	\end{table*}
	
    \subsection{Logistic regression experiments}
	We solve the following minimization problem:
	\[
	\min_{x \in \mathbb{R}^d} \left\{ f(x) \eqdef \frac{1}{m} \sum_{i=1}^m  \log\left(1-e^{-b_i a_i^\top x}\right) + \frac{\mu}{2} \|x\|_2^2 \right\}.
	\]
	To make problem and data balanced, we normalise each data point and get $\|a_i\|_2=1$ for every $i \in [1,\ldots, m]$. Parameters of all methods are fine-tuned to get the fastest convergence. Note, that it is possible that for bigger $L$ method converges faster in practice.
	
	\begin{figure}[!h]
	%\centering
	\includegraphics[width=0.32\linewidth]{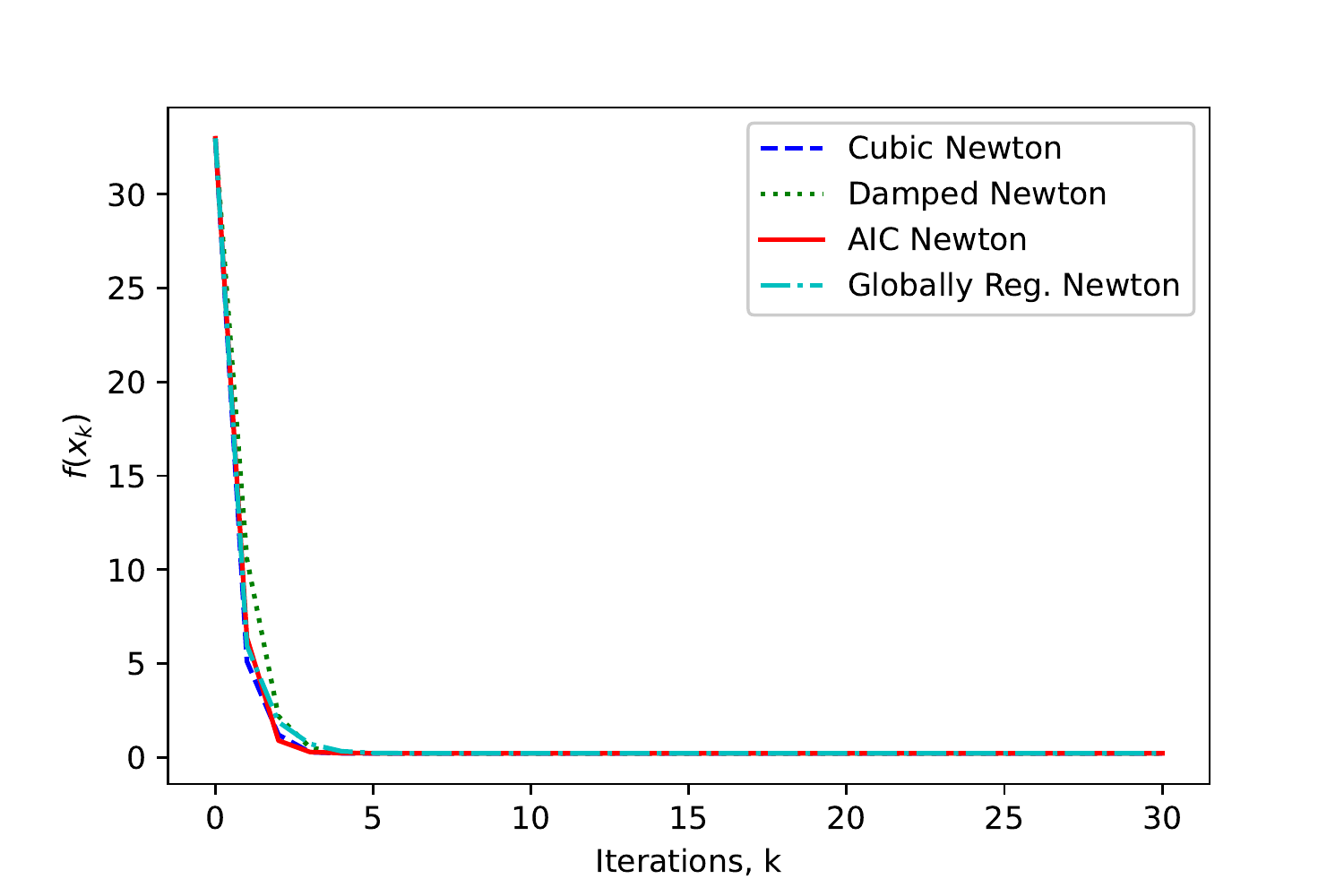}
	\includegraphics[width=0.32\linewidth]{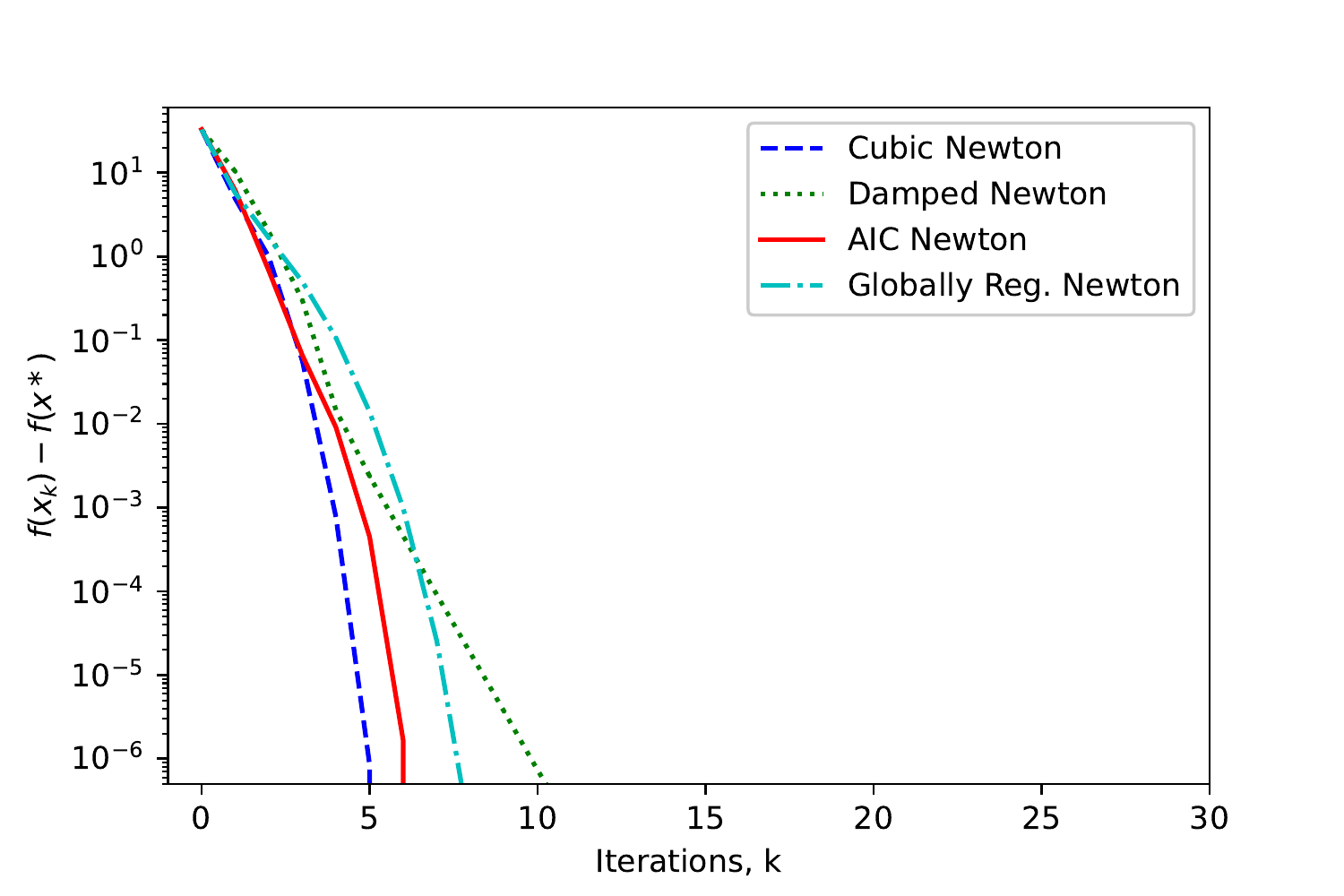}
	\includegraphics[width=0.32\linewidth]{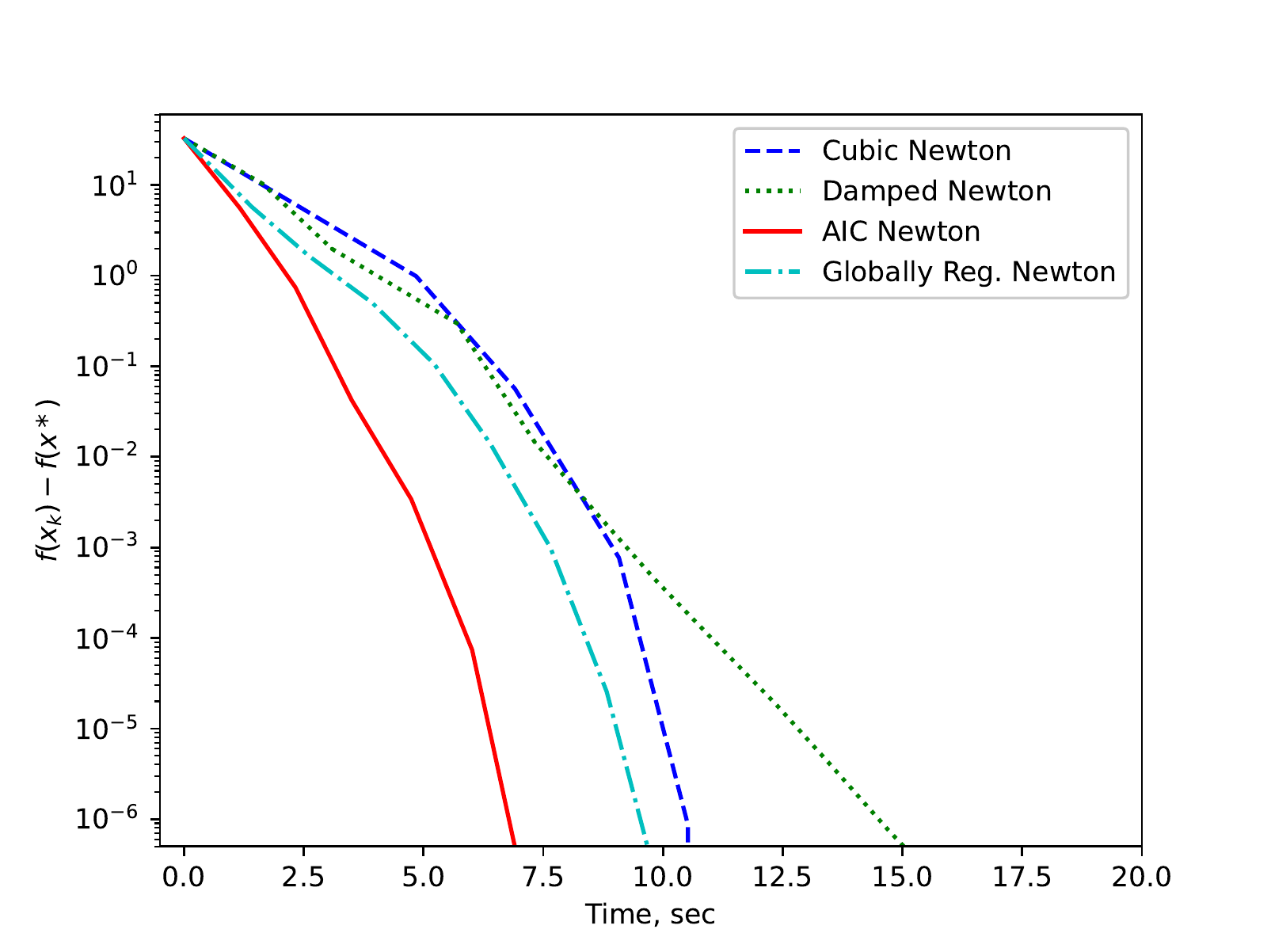}
	\caption{Comparison of regularized Newton methods and Damped Newton method for logistic regression task on  \emph{w8a} dataset.}
	\label{fig:w8a}
	\vspace{-0.25cm}
	\end{figure}
	
	In \Cref{fig:w8a}, we consider classification task on dataset \emph{w8a} \citep{chang2011libsvm}. Number of features for every data sample is $d=300$, $m=49749$. We take starting point $x_0 \eqdef 8 [1,1, \dots, 1]^\top$ and $\mu = 10^{-3}$. Fine-tuned values are $\Lalg = 0.6$, $L_2= 0.0001$, $\alpha = 0.5$. We can see that all methods are very close.  Damped Newton has rather big step $0.5$, so it is fast at the beginning but still struggle at the end because of the fixed step-size.

	\begin{figure}[t]
	%\centering
	\includegraphics[width=0.32\linewidth]{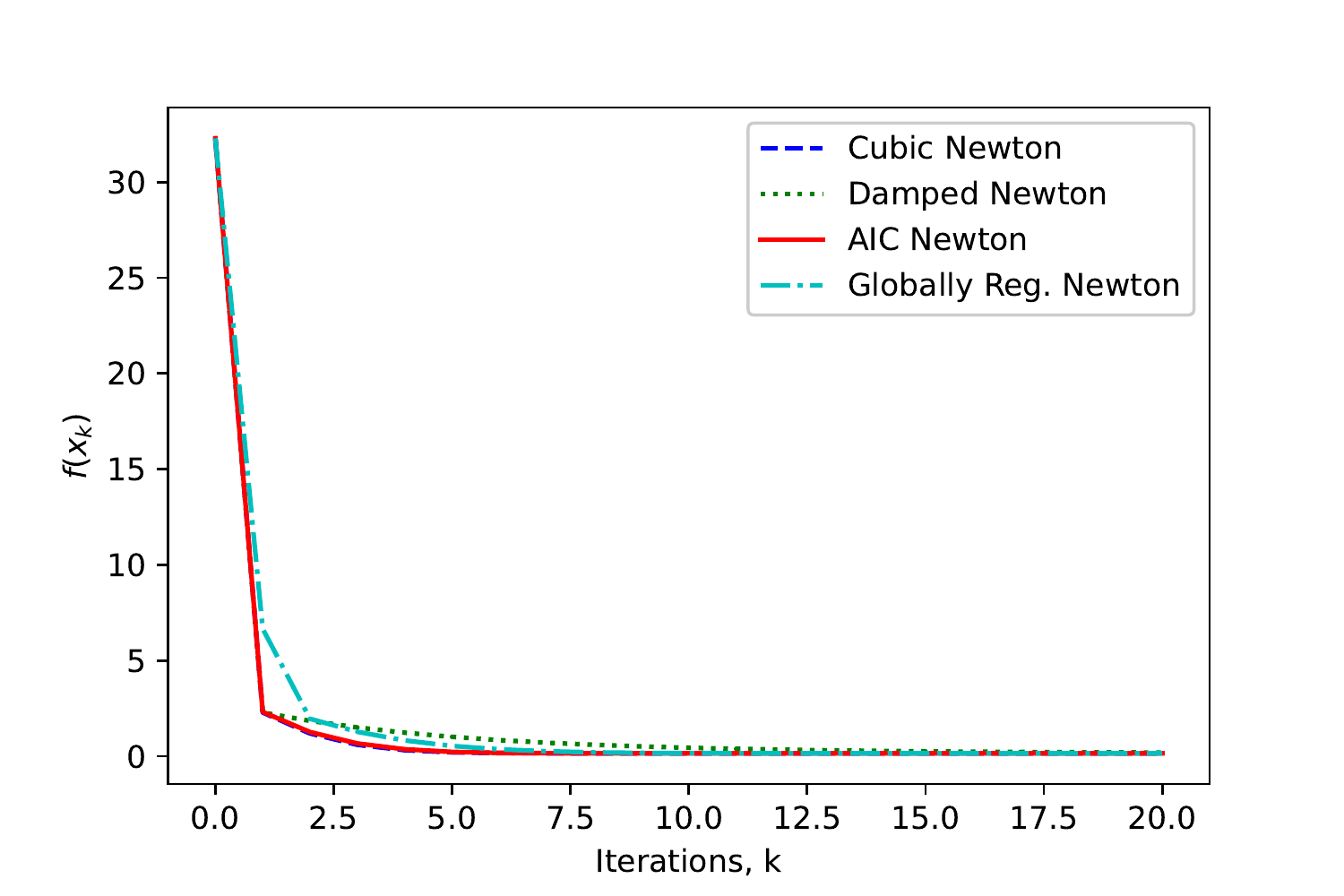}
	\includegraphics[width=0.32\linewidth]{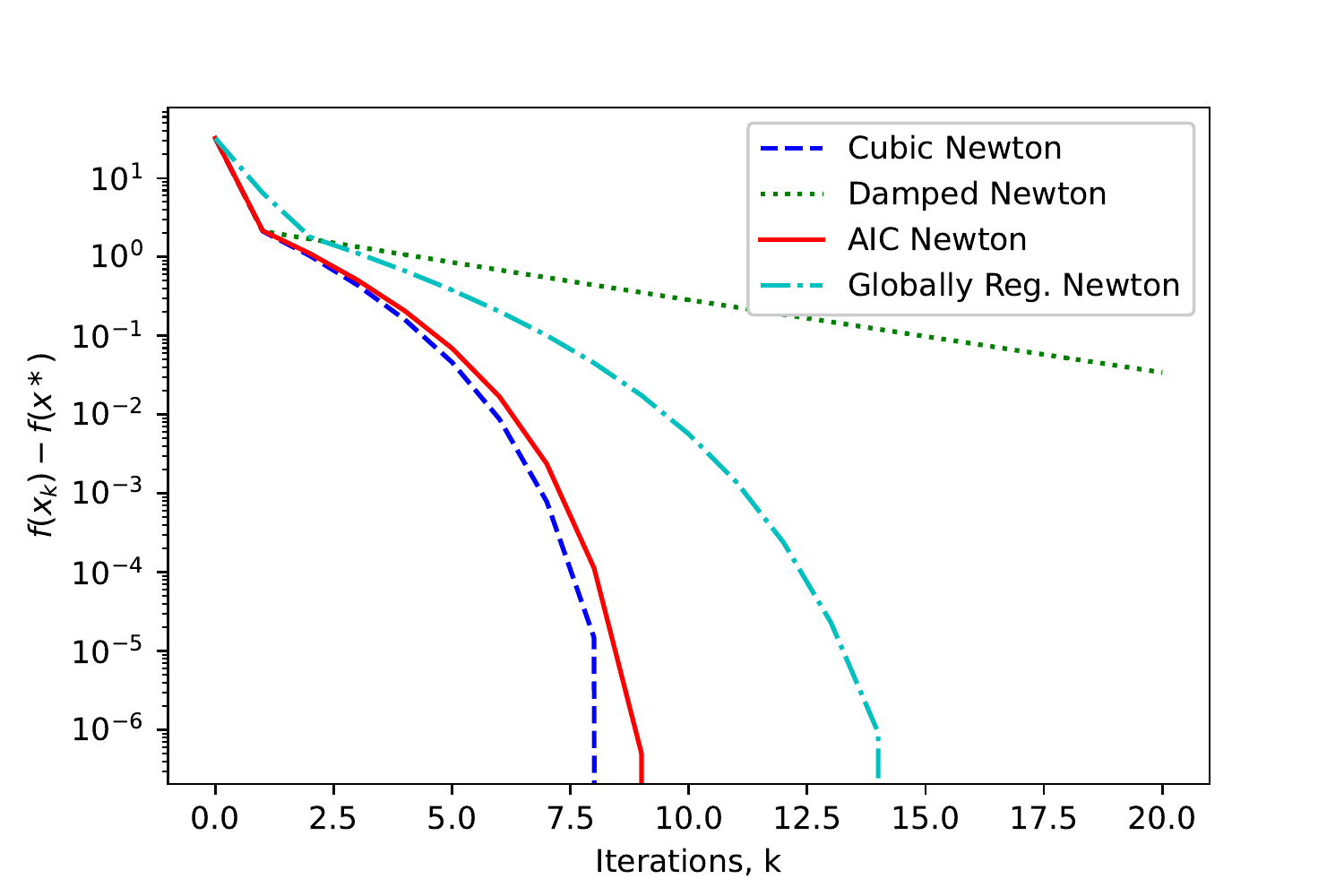}
	\includegraphics[width=0.32\linewidth]{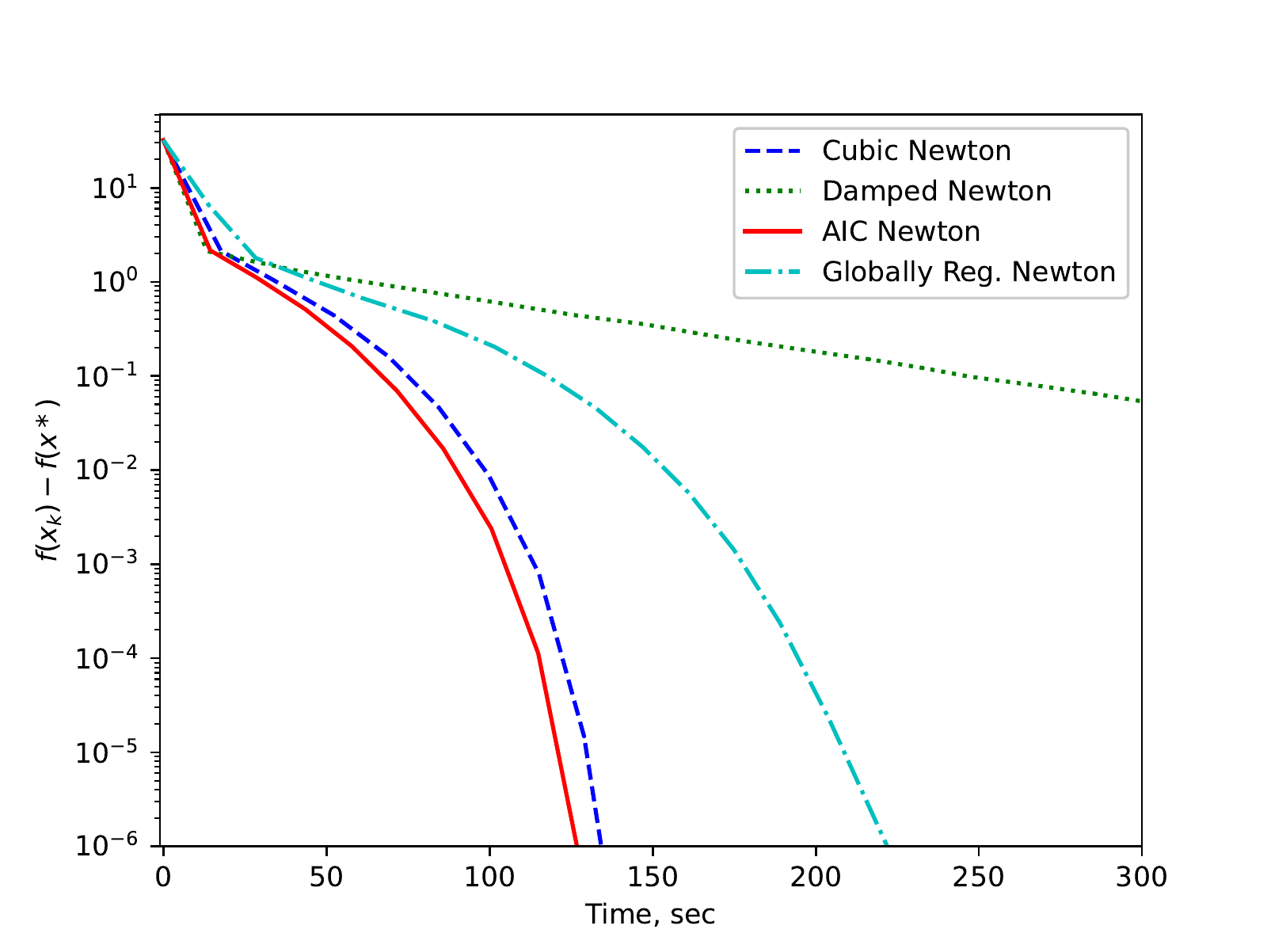}
	\caption{Comparison of regularized Newton methods and Damped Newton method for logistic regression task on  \emph{MNIST} dataset (0 vs. all other digits).}
	\label{fig:mnist}
	\vspace{-0.25cm}
	\end{figure}
	
	In \Cref{fig:mnist}, we consider binary classification task on dataset \emph{MNIST} \citep{deng2012mnist} (one class contains images with 0, another one~--- all others). Number of features for every data sample is $d=28^2=784$, $m=60000$. We take starting point $x_0 \eqdef 3 \cdot [1,1, \dots, 1]^\top$ (such that Newton method started from this point diverges) and $\mu = 10^{-3}$. Fine-tuned values are $\Lalg = 10$, $L_2 = 0.0003$ for Globally Reg. Newton and Cubic Newton, $\alpha = 0.1$. We see that \ain{} has the same iteration convergence as Cubic Newton but faster by time.
	
	\begin{figure}[H]
		\centering
		\includegraphics[width=0.32\textwidth]{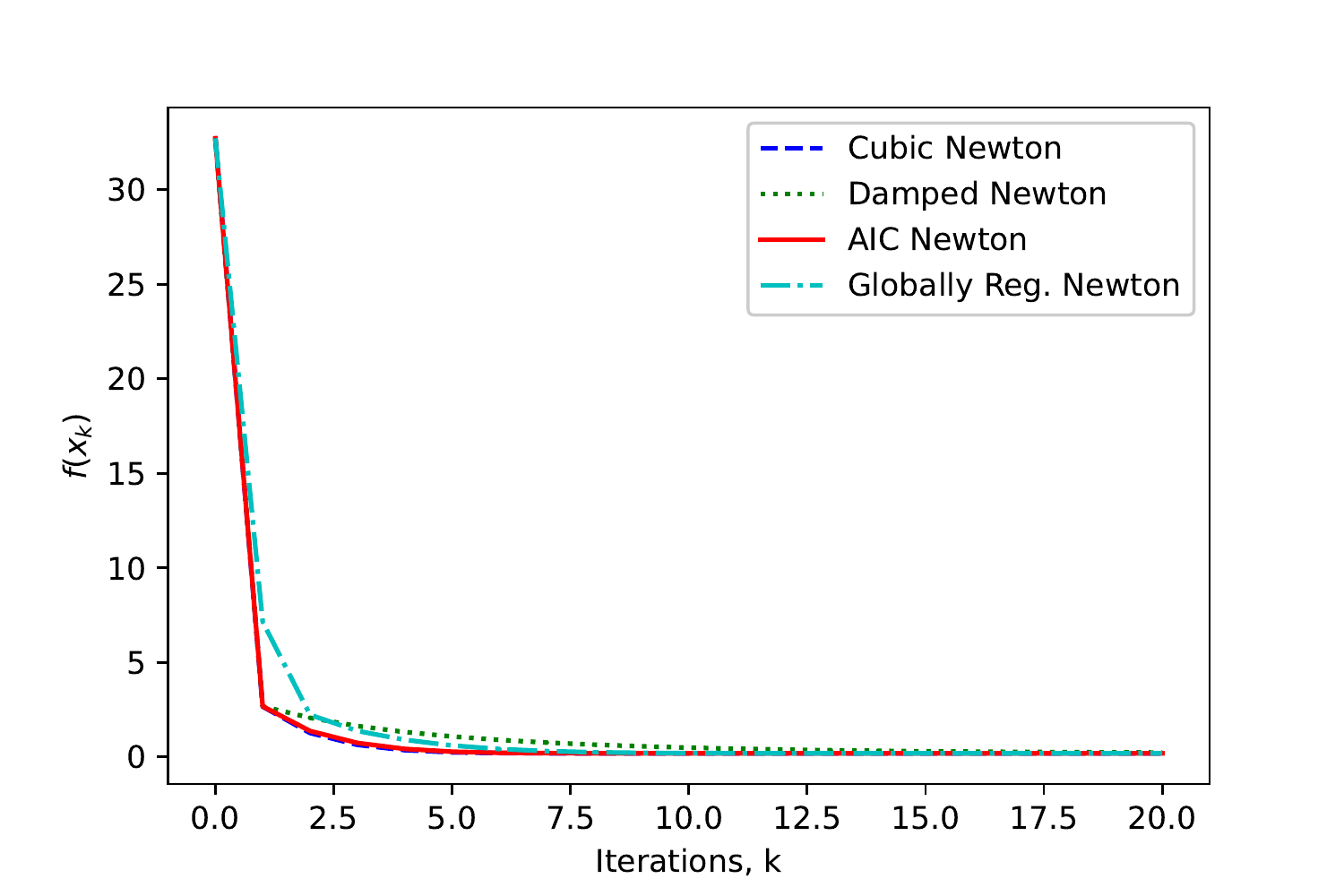}
% 		\hspace{0.02\textwidth}
		\includegraphics[width=0.32\textwidth]{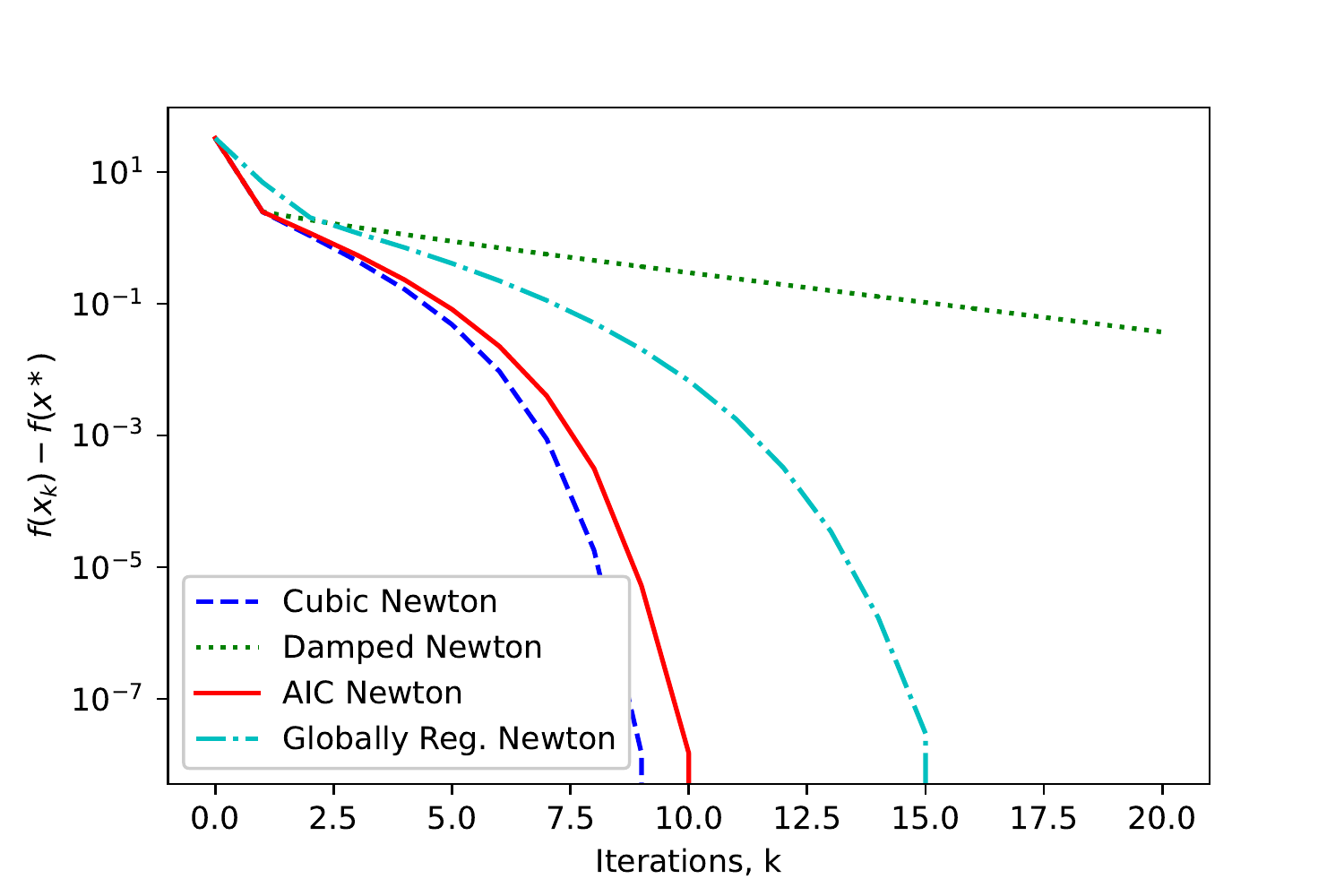}
		\includegraphics[width=0.32\textwidth]{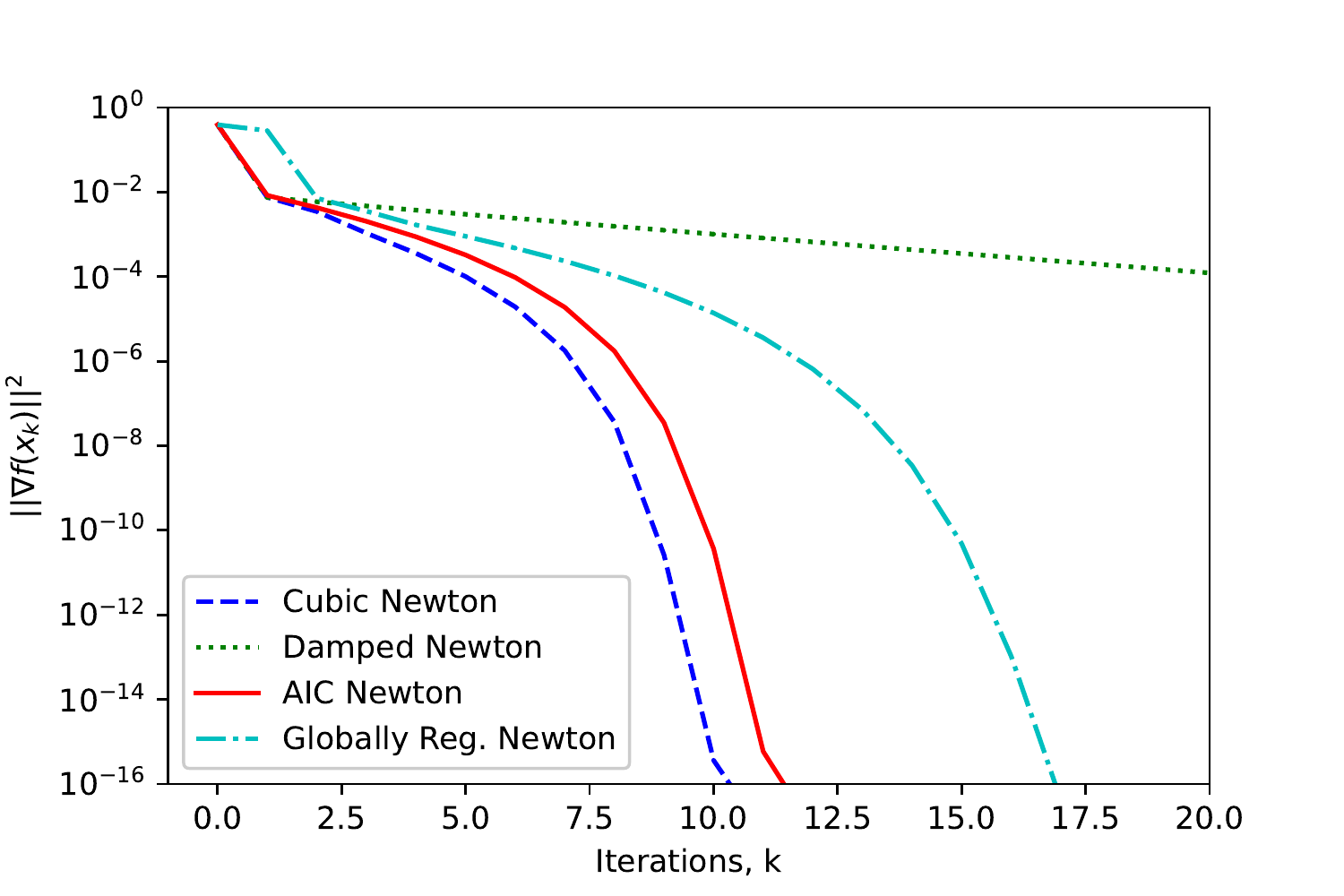}
		\caption{Comparison of regularized Newton methods and Damped Newton method for logistic regression task on \emph{MNIST} dataset (10 models for $i$ vs. other digits problems with argmax aggregation).}
		\label{fig:mnist_all}
	\end{figure}
	
	In \Cref{fig:mnist_all}, we present the results for multi-class classification problem on dataset \emph{MNIST}. We train 10 different models in parallel, each one for the problem of binary classification distinguishing $i$-th class out of others. Loss on current iteration for the plots is defined as average loss of $10$ models. Prediction is determined by the maximum ``probability'' predicted by $i$-th model.
	The estimates for the parameters of methods are the same as in previous experiment.
	We see that \ain{} is the same speed as Cubic Newton and Globally Reg. Newton and much faster than Damped Newton in both value of function and gradient norm.

	For normalized problem, we can analytically compute an upper bound for theoretical constant $L_2$
	\[\|\nabla^3 f(x)\|_{2}\leq L_2. \]
	One can show that $L_2= \tfrac{\sqrt{3}}{18}\simeq 0.1$. In our experiments, we show that Cubic Newton can work with much lower constants:
	\textit{madelon} - $0.015$, \textit{w8a} - $0.00003$, \textit{a9a} - $0.000215$. It means that theoretical approximation of the constants can be bad and we have to tune them for all methods. 

	\section{Proofs of Results Appearing in the Paper} \label{sec:ap_theory}
	In this section, we present proofs of the lemmas and theorems from the main paper's body.

	\subsection{Proofs regarding  affine invariance (\Cref{sec:affine_invariance})} \label{ap:affine-invariance}
	\begin{proof}[Proof of \Cref{le:AI_newton_nesterov} (Lemma 5.1.1, \citet{nesterov2018lectures})] \textbf{[Newton method is affine-invariant]}
		
		Let $y_k = \mathbf A^{-1} x_k$ for some $k\geq 0$ and $\alpha_k$ be affine-invariant. Firstly,
		\begin{align*}
		y_{k+1} &= y_{k} - \alpha_k\left[\nabla^2 \phi(y_k) \right]^{-1} \nabla \phi(y_k)
		= y_{k} - \alpha_k\left[\mathbf A^\top\nabla^2 f(\mathbf A y_k)\mathbf A \right]^{-1} \mathbf A^\top \nabla f(\mathbf A y_k)\\
		&= \mathbf A^{-1}x_{k} - \alpha_k \mathbf A^{-1}\left[\nabla^2 f(x_k)\mathbf A \right]^{-1} \nabla f(x_k) = \mathbf A^{-1}x_{k+1}.
		\end{align*}
		
		Secondly note that $\normMd {\g(x)} {x}$ is affine invariant, as 
		\[\normMd{\nabla g(y_k)} {y_k} 
		=  \nabla g(y_k)^\top \nabla^2 g(y_k)^{-1} \nabla g(y_k) 
		= \nabla f(x_k)^\top \nabla^2 f(x_k)^{-1} \nabla f(x_k) 
		= \normMd{\nabla f(x_k)} {x_k}. \]
		Consequently, stepsizes $\alpha_k$ from \ain{} \eqref{eq:update} and \citet{nesterov2018lectures} are all affine-invariant. Hence these damped Newton algorithms are affine-invariant.
	\end{proof}
	
	\begin{proof} [Proof of \Cref{le:model}] \textbf{[Upper bound from from semi-strong self-concordance]}
		
		We rewrite function value approximation from the left hand side as
		\begin{align*}
		f(y) - f(x) - \nabla f(x)[y-x]
		&= \int\limits_0^1  \ls\nabla f(x+\tau(y-x))-\nabla f(x)\rs[y-x] d \tau \\
		&= \int\limits_0^1 \int\limits_0^{\tau}   \ls\nabla^2 f(x+\lambda(y-x))\rs[y-x]^2 d \lambda d \tau.
		\end{align*}
		
		Taking its norm, we can finish the proof as
		\begin{align*}
		&\left|f(y) - f(x) - \nabla f(x)[y-x] - \frac{1}{2}\nabla^2 f(x)[y-x]^2\right| \\
		& \qquad = \left|\int\limits_0^1 \int\limits_0^{\tau}   \ls\nabla^2 f(x+\lambda(y-x))-\nabla^2 f(x)\rs[y-x]^2 d \lambda d \tau \right|\\
		& \qquad \leq \int\limits_0^1 \int\limits_0^{\tau} \left|  \ls\nabla^2 f(x+\lambda(y-x))-\nabla^2 f(x)\rs[y-x]^2 \right|d \lambda d \tau \\
		& \qquad \stackrel{\eqref{eq:semi-strong-self-concordance}}{\leq} \int\limits_0^1 \int\limits_0^{\tau} \Lsemi \lambda \|y-x\|_x^3 d \lambda d \tau
		= \frac{\Lsemi}{6} \|y-x\|_x^3 .
		\end{align*}
	\end{proof}

	\begin{proof}[Proof of \Cref{th:stepsize}] \textbf{[Minimizer of the model \eqref{eq:S_step} has form of damped Newton method \eqref{eq:update}]}
		
		Proof is straightforward. To show that \ain{} model update minimizes $S_{f,\Lalg}(x)$, we compute the gradient of $S_{f,\Lalg}(x)$ at next iterate of \ain{}. Showing that it is $0$ concludes $x_{k+1} = S_{f,\Lalg}(x_k)$.
		
		For simplicity, denote $h \eqdef y-x$. We can simplify the implicit update step $S_{f,\Lalg}(x)$
		\begin{align}
		S_{f,\Lalg}(x) 
		&= \argmin_{y\in \bbE} \left\lbrace f(x) + \la \nabla f(x),y-x\ra + \frac{1}{2} \la \nabla^2 f(x)(y-x),y-x\ra +\frac{\Lalg}{6}\|y-x\|_{x}^3\right\rbrace\\
		&= x + \argmin_{h  \in \bbE} \left\lbrace \la \nabla f(x),h\ra + \frac 1 2 \normsM h x. +\frac{\Lalg}{6}\|h\|_{x}^3\right\rbrace.
		\end{align}
		Taking gradient of the subproblem with respect to $h$,
		\begin{align}
		\nabla_h \left( \la \nabla f(x),h\ra + \frac 1 2 \normsM h x +\frac{\Lalg}{6}\|h\|_{x}^3 \right)= \g(x) + \h(x) h + \frac {\Lalg} 2 \h(x)h \normM h x.
		\end{align}
		and setting $h$ according to \ain{}, $h = -\alpha \h(x)^{-1} \g(x)$, leads to
		\begin{align} 
		\g(x) -\alpha \g(x) - \frac {\Lalg} 2 \alpha^2 \g(x) \normM {\h(x)^{-1} \g(x)} x
		=-\g(x) \left( -1 + \alpha + \frac {\Lalg} 2 \alpha^2 \normMd {\g(x)} x \right).
		\end{align}

		Finally, \ain{} stepsize $\alpha$ \eqref{eq:update} is chosen as a root of quadratic function
		\begin{equation} \label{eq:alpha_star}
		\frac {\Lalg} 2 \normMd {\g(x)} x \alpha^2 + \alpha -1 =0,
		\end{equation}
		hence the gradient of the \eqref{eq:step} at next iterate of \ain{} is $0$. This concludes the proof.
	\end{proof}
	
	\subsection{Proof of global convergence (\Cref{sec:global})}
	\label{ap:global}
	\begin{proof} [Proof of \Cref{lm:main_lemma}] \textbf{[One step decrease globally under semi-strong self-concordance]}
	
		This claim follows directly from \Cref{le:model}.
		\begin{align*}
		f(S_{f,\Lalg}(x)) 
		&\stackrel{\eqref{eq:wssc_ub}}{\leq} f(x) + \la \nabla f(x), S_{f,\Lalg}(x)-x\ra + \frac{1}{2} \la \nabla^2 f(x)( S_{f,\Lalg} (x)-x),S_{f,\Lalg}(x)-x\ra \\
		& \qquad +\frac{\Lalg}{6}\| S_{f,\Lalg}(x)-x\|_{x}^3 \\
		&\stackrel{\eqref{eq:step}}{=}  \min \limits_{y \in \bbE} \left \{ f(x) + \la \nabla f(x),y-x\ra + \frac{1}{2} \la \nabla^2 f(x)(y-x),y-x\ra +\frac{\Lalg}{6}\|y-x\|_{x}^3 \right \} \\
		&\stackrel{\eqref{eq:lower_upper_bound}}{\leq}  \min \limits_{y \in \bbE} \left\{ f(y)   +  \tfrac{\Lalg}{3}\|y-x\|_{x}^3 \right\} .  
		\end{align*}
	\end{proof}
	
	\begin{proof} [Proof of \Cref{thm:convergence}] \textbf{[Global convergence under semi-strong self-concordance]}
		
		We start by taking Lemma \ref{lm:main_lemma} for any $t \geq 0$, we obtain
		\begin{eqnarray*}
		f(x_{t+1}) 
		&\stackrel{\eqref{eq:min_lemma}}{\leq}& 
		\min \limits_{y \in \bbE} \left\{ f(y)  + \tfrac{\Lalg}{3}\|y-x_t\|_{x_t}^3   \right \}  \\
		&\stackrel{\eqref{D_lebeg}}{\leq}&  \min \limits_{\eta_t \in [0, 1]} \left \{ f(x_{t} + \eta_t (x_{\ast} - x_t)) + \tfrac{\Lalg}{3}\eta_t^3 D^3  \right\} \\
		&\leq & \min \limits_{\eta_t \in [0, 1]} \left \{ (1-\eta_t)f(x_t) + \eta_t f(x_{\ast}) + \tfrac{\Lalg}{3}\eta_t^3 D^3 \right\},
		\end{eqnarray*}
		where for the second line we take $y=x_t + \eta_t (x_{\ast}-x_t)$ and use convexity of $f$ for the third line.
		Therefore, subtracting $f(x_{\ast})$ from both sides, we obtain, for any $\eta_t \in [0,1]$
		\begin{equation}\label{eq:conv_to_sum}
		f(x_{t+1}) - f(x_{\ast}) \leq (1-\eta_t)(f(x_t) - f(x_{\ast})) +\tfrac{\Lalg}{3}\eta_t^3 D^3 
		\end{equation}
		Let us define the sequence $\{A_t\}_{t\geq 0}$ as follows:
		\begin{equation} \label{A_t}
		\gboxeq{
		A_t \eqdef 
		\begin{cases}
		1 , & t = 0\\
		\prod \limits_{i=1}^t (1-\eta_i), & t\geq 1.
		\end{cases}
		}
		\end{equation}
		Then $A_t = (1-\eta_t)A_{t-1}$. Also, we define \gbox{$\eta_0 \eqdef 1$.}
		Dividing both sides of \eqref{eq:conv_to_sum} by $A_t$, we get
		\begin{align}
		\frac{1}{A_t}\left(f(x_{t+1}) - f(x_{\ast})\right)
		&\leq \frac{1}{A_t} \left(1-\eta_t)(f(x_t) - f(x_{\ast}) \right)
		+ \frac{\eta_t^3}{A_t} \frac{\Lalg D^3}{3} \notag\\
		&= \frac{1}{A_{t-1}}(f(x_t) - f(x_{\ast})) 
		+ \frac{\eta_t^3}{A_t} \frac{\Lalg D^3}{3}  \label{eq:summing}.
		\end{align}
		Summing  both sides of inequality \eqref{eq:summing} for $t = 0, \ldots, k$ , we obtain
		\begin{eqnarray*}
		\frac{1}{A_k}(f(x_{{k}+1}) - f(x_{\ast}))  
		&\leq& \frac{(1-\eta_0)}{A_0}(f(x_0) - f(x_{\ast}))
		+ \frac{\Lalg D^3}{3}  {\sum \limits_{t=0}^k \frac{ \eta_t ^{3}}{A_t}} \notag\\
		&\stackrel{\eta_0=1}{=} &
		\frac{\Lalg D^3}{3}  {\sum \limits_{t=0}^k \frac{ \eta_t ^{3}}{A_t}}.
		\end{eqnarray*}
		As a result, 
		\begin{equation}
		f(x_{{k}+1}) - f(x_{\ast}) \leq
		\frac{\Lalg D^3}{3}  {\sum \limits_{t=0}^k \frac{A_k \eta_t ^{3}}{A_t}}p \label{eq:last} 
		\end{equation}
		To finish the proof, we need to choose $\eta_t$ so that $\sum_{t=0}^{k} \frac{A_{k} \eta_{t}^{3}}{A_{t}} = O(k^{-2})$. This holds for\footnote{Note that formula of $\eta_0$ coincides with its definition above.}
		\begin{equation} \label{alpha_t}
		\gboxeq{\eta_t \eqdef \frac{3}{t+3}, \quad t \geq 0,}
		\end{equation}
		as stated in the next lemma.
		
		\begin{lemma}[Properties of $\eta_t$ and $A_t$, [from (2.23), \citet{ghadimi2017second}]] \label{le:sequence_bound}
		    For choice $\eta_t$ as \eqref{alpha_t} and $A_t$ as \eqref{A_t} we have
		    \begin{align}
		        A_t &= \frac 6 {(t+1)(t+2)(t+3)},\\
		        \sumin tk \frac {\eta_t^3} {A_t} &= \sumin tk \frac {9(t+1)(t+2)}{2(t+3)^2} \leq \frac {9k} 2.
		    \end{align}
		\end{lemma}
		Plugging \Cref{le:sequence_bound} inequalities to \eqref{eq:last} concludes the proof of the \Cref{thm:convergence},
		\begin{align*}
		f(x_{{k}+1}) - f(x_{\ast}) 
		\leq \frac 6 {(k+1)(k+2)(k+3)} \frac{\Lalg D^3}{3} \frac {9k} 2
		\leq \frac{9\Lalg  D^3}{k^2}
		\leq \frac{9\Lalg  R^3}{k^2}.
		\end{align*}

	\end{proof}
	
	For readers convenience, we include proof of \Cref{le:sequence_bound} in \Cref{ssec:technical}.
	
	\subsection{Proofs of local convergence (\Cref{sec:local})}
	
	\begin{proof}[Proof of \Cref{le:lsconv}] \textbf{[Strong convexity / Bound on inverse Hessian norm change]}
	
	Claim follows from Theorem 5.1.7 of \citet{nesterov2018lectures}, which states that for $\Lstandard$-self-concordant function, hence also for $\Lsemi$-semi-strongly self-concordant function $f$ and $x_k, x_{k+1}$ such that $\frac {\Lsemi} 2 \normM {\xdiff } {x_k} < 1$ holds
		\begin{equation*}
		\left( 1- \frac {\Lsemi} 2 \normM {\xdiff } {x_k} \right)^2 \h(x_{k+1}) \preceq \h(x_k) \preceq \left( 1- \frac {\Lsemi} 2 \normM {\xdiff } {x_k} \right)^{-2} \h(x_{k+1}).
		\end{equation*}
	Let $c$ be some constant, $0<c<1$.
        Then for updates of \ain{} in the neighborhood $$\left \{ x_k: c \geq \frac{\Lalg}{2}\alpha_k \normMd {\g (x_k)} {x_k} \right \}$$ holds
		\begin{equation} \label{eq:inv_hessian_bound_scf}
		\h(x_{k+1})^{-1} \preceq \left( 1 - c \right)^{-2} \h(x_k)^{-1}.
		\end{equation}
    \end{proof}
    In order to prove \Cref{le:one_step_local}, we first use semi-strong self-concordance to prove a key inequality -- a version of Hessian smoothness, bounding gradient approximation by difference of points.
    \begin{lemma}\label{le:sssc_to_loc_bound}
    	For semi-strongly self-concordant function $f$ holds
    	\begin{equation}
    	\label{eq:semi_norm_ineq_}
    	\left\|\nabla f(y) - \nabla f(x) - \nabla^2 f(x)[y-x] \right\|^*_x
    	\leq \frac{\Lsemi}{2} \|y-x\|_x^2.
    	\end{equation}
    \end{lemma}
    \begin{proof} [Proof of \Cref{le:sssc_to_loc_bound}] \textbf{[Local smoothness assumption follows from semi-strong self-concordance]}
	
	We rewrite gradient approximation on the left hand side as
	\begin{gather*}
	\nabla f(y) -  \nabla f(x) - \nabla^2 f(x)[y-x] =
	\int\limits_0^1  \ls\nabla^2 f(x+\tau(y-x))-\nabla^2 f(x)\rs[y-x] d \tau.
	\end{gather*}
	Now, we can bound it in dual norm as
	\begin{align*}
	\left\|\nabla f(y) - \nabla f(x) - \nabla^2 f(x)[y-x] \right\|_x^{\ast}
	&= \left\|\int\limits_0^1   \ls\nabla^2 f(x+\tau(y-x))-\nabla^2 f(x)\rs[y-x] d \tau \right\|_x^{\ast}\\
	&\leq \int\limits_0^1 \left\|  \ls\nabla^2 f(x+\tau(y-x))-\nabla^2 f(x)\rs[y-x] \right\|_x^{\ast}  d \tau \\
	&\leq \int\limits_0^1  \left\|\nabla^2 f(x+\tau(y-x))-\nabla^2 f(x)\right\|_{op} \|y-x\|_x   d \tau \\
	&\stackrel{\eqref{eq:semi-strong-self-concordance}}{\leq} \int\limits_0^1  \Lsemi \tau \|y-x\|_x^2 d  \tau
	= \frac{\Lsemi}{2} \|y-x\|_x^2,
	\end{align*}
	where in second inequality we used property of operator norm \eqref{eq:matrix_operator_norm}.
\end{proof}	
      Finally, we are ready to prove one step decrease and the convergence theorem.
	
	\begin{proof}[Proof of \Cref{le:one_step_local}]\textbf{[One step decrease locally under semi-strong self-concordance]}
		
		We bound norm of $\nabla f(x_{k+1})$ using basic norm manipulation and triangle inequality as
		\begin{align*}
		\normMd {\nabla f(x_{k+1})} {x_k}
		&\stackrel{\eqref{eq:update}}{=} \normMd{\nabla f(x_{k+1}) - \nabla^2 f(x_k)(x_{k+1}-x_k)  - \alpha_k \nabla f(x_k) }{x_k}\\
		&= \normMd{\nabla f(x_{k+1}) - \nabla f(x_{k}) - \nabla^2 f(x_k)(x_{k+1}-x_k)  + (1-\alpha_k) \nabla f(x_k) }{x_k}\\
		&\leq \normMd{\nabla f(x_{k+1}) - \nabla f(x_{k}) - \nabla^2 f(x_k)(x_{k+1}-x_k)}{x_k}  + (1-\alpha_k) \normMd{\nabla f(x_k) }{x_k}
		\end{align*}
		
		Using \Cref{le:sssc_to_loc_bound}, 
		we can continue
		\begin{align*}
		\normMd {\nabla f(x_{k+1})} {x_k}
		&\leq \normMd{\nabla f(x_{k+1}) - \nabla f(x_{k}) - \nabla^2 f(x_k)(x_{k+1}-x_k)}{x_k}  + (1-\alpha_k) \normMd{\nabla f(x_k) }{x_k}  \\
		&\stackrel{\eqref{eq:semi_norm_ineq_}}{\leq} 
		\frac{\Lsemi}{2} \normsM{x_{k+1}-x_k} {x_k} + (1-\alpha_k) \normMd{\nabla f(x_k) }{x_k} \\
		&\stackrel{\eqref{eq:update}}{\leq} 
		\frac{\Lsemi \alpha_k^2}{2} \normsMd{\nabla f(x_k)} {x_k} + (1-\alpha_k) \normMd{\nabla f(x_k) }{x_k} \\
		&\leq  \frac{\Lalg \alpha_k^2}{2} \normsMd{\nabla f(x_k)} {x_k} + (1-\alpha_k) \normMd{\nabla f(x_k) }{x_k} \\
		&=  \left(\frac{\Lalg \alpha_k^2}{2} \normMd{\nabla f(x_k)} {x_k} -\alpha_k +1\right) \normMd{\nabla f(x_k) }{x_k}\\
		&\stackrel{\eqref{eq:alpha_star}}{=} \Lalg \alpha_k^2\normsMd{\nabla f(x_k) }{x_k}
		\end{align*}
		
		We use \Cref{le:lsconv} to shift matrix norms.
		\begin{eqnarray}
		\normMd{\g (x_{k+1})}{x_{k+1}}
		& \stackrel{\eqref{eq:inv_hessian_bound_scf}} \leq & \frac{1}{1-c} \normMd{\g (x_{k+1})} {x_k} \notag \\
		& \stackrel{\eqref{eq:one_step_local}} \leq & \frac{\Lalg \alpha_k^2}{1-c}  \normsMd{\g (x_k)} {x_k}\label{eq:local_shifted_scf_alt}\\
		& < & \frac{\Lalg \alpha_k}{1-c}  \normsMd{\g (x_k)} {x_k}. \notag
		\end{eqnarray}
		
		We obtain neighborhood of decrease by solving $ \frac{\Lalg \alpha_k}{1-c}  \normMd{\g (x_k)} {x_k} \leq 1$ in $\normMd{\g  (x_k)} {x_k}$.
		
	\end{proof}
	
	\begin{proof} [Proof of \Cref{th:local}] \textbf{[Local convergence under semi-strong self-concordance]}
		
		Let $c = \frac{1}{3}$,
		then for $\normMd{\g (x_{0})} {x_{0}} < \frac {8 } {9\Lalg}$, we have
		$\frac{\Lalg \alpha_0}{1-c}  \normMd{\g (x_0)} {x_0} \leq 1$ and $c \geq \frac{\Lalg}{2}\alpha_0 \normMd {\g (x_0)} {x_0}$. Then from \Cref{le:one_step_local} we have  guaranteed the decrease of gradients $\normMd \gn {x_{k+1}} \leq \normMd \gk {x_k} < \frac {8} {9 \Lalg}.$ We finish proof by chaining \eqref{eq:local_shifted_scf_alt} and simplifying with $\alpha_i \leq 1$.
		\begin{align}
			\normMd{\g(x_{k})}{x_{k}}
			&\leq \left( \tfrac {3} {2} \Lalg \right )^k \left( \prod _{i=0}^{k} \alpha_i^{2} \right)
			\left( \normMd{\g(x_{0})} {x_{0}} \right)^{2^k} .
		\end{align}
	\end{proof}

	\subsection{Technical lemmas} \label{ssec:technical}
	\begin{lemma}[Arithmetic mean -- Geometric mean inequality]
		For $c\geq0$ we have 
		\begin{equation} \label{eq:AG}
		1+ c = \frac {1 + (1+2c)}{2} \stackrel{AG} \geq \sqrt{1+2c}.
		\end{equation}
	\end{lemma}
	
	\begin{lemma}[Jensen for square root]
		Function $f(x) = \sqrt x$ is concave, hence for $c\geq0$ we have 
		\begin{equation}\label{eq:jensen}
		\frac 1 {\sqrt 2} (\sqrt c+1) \leq \sqrt{c+1} \qquad \leq \sqrt c + 1.
		\end{equation}
	\end{lemma}
	
	\noindent
	\textbf{\Cref{le:sequence_bound}} [(2.23) from \citep{ghadimi2017second}] For 
	\[ \eta_t \eqdef \frac{3}{t+3}, \quad t \geq 0, \qquad \text{and} \qquad A_t \eqdef 
		\begin{cases}
		1 , & t = 0\\
		\prod \limits_{i=1}^t (1-\eta_i), & t\geq 1
		\end{cases}
		\]
	we have
    \begin{align}
        A_t &= \frac 6 {(t+1)(t+2)(t+3)} \qquad \text{and} \qquad \sumin tk \frac {\eta_k^3} {A_t} &= \sumin tk \frac {9(t+1)(t+2)}{2(t+3)^2} \leq \frac {3k} 2.
    \end{align}
	
	\begin{proof}[Proof of \Cref{le:sequence_bound}]
		We have
		\begin{equation}\label{eq:A_t_bound}
		A_{k} =\prod_{t=1}^{k}\left(1-\eta_{t}\right)=\prod_{t=1}^{k} \frac{t}{t+3}=\frac{k! \, 3!}{(k+3) !}= 3! \prod_{j=1}^{3} \frac{1}{k+j}, 
		\end{equation}
		which gives,
		\begin{equation}
		\label{eq:sum_A_t_bound}
		\begin{aligned}
		\sum_{t=0}^{k} \frac{A_{k} \eta_{t}^{3}}{A_{t}} &=\sum_{t=0}^{T} \frac{3^{3}}{(t+3)^{3}} \prod_{j=1}^{3}\frac{t+j}{k+j} =3^{3} \prod_{j=1}^{3} \frac{1}{k+j} \sum_{t=0}^{k}  \frac{\prod_{j=1}^{3} (t+j)}{(t+3)^3}.
		\end{aligned}
		\end{equation}
		The sum is non-decreasing. Indeed, we have
		\begin{equation*}
		1 \leq 1 +  \frac{1}{t+3} \leq  1 + \frac{1}{t+j}, \quad \forall j \in \{1,2, 3\},
		\end{equation*}
		and, hence, for all $j \in \{1,2,3\}$,
		\begin{eqnarray*}
		&\ls 1 +  \frac{1}{t+3}\rs^3 &\leq \prod_{j=1}^{3} \ls 1 + \frac{1}{t+j} \rs \\
		\Leftrightarrow & \ls\frac{t+4}{t+3}\rs^3 & \leq \prod_{j=1}^{3} \frac{t+j+1}{t+j}\\
		\Leftrightarrow &  \frac{\prod_{j=1}^{3} (t+j)}{(t+3)^3} & \leq \frac{\prod_{j=1}^{3} (t+1+j)}{(t+1+3)^3} .
		\end{eqnarray*}
		
		Thus, we have shown that the summands in the RHS of \eqref{eq:sum_A_t_bound} are growing, whence we get the next upper bound for the sum
		\begin{equation}
		\label{eq:ATat/At_i0}
		\begin{aligned} 
		\sum_{t=0}^{k} \frac{A_{k} \eta_{t}^{3}}{A_{t}} 
		&=3^{3} \prod_{j=1}^{3} \frac{1}{k+j} \sum_{t=0}^{k}  \frac{\prod_{j=1}^{3} (t+j)}{(t+3)^3}\\
		&\leq 3^3 \prod_{j=1}^{3} \frac{1}{k+j}\cdot (k+1) \cdot  \frac{ \prod_{j=1}^{3} (k+j)}{(k+3)^3}
		\leq\frac{(k+1)3^{3}}{(k+3)^{3}}
		\leq O\ls \frac{1}{k^{2}}\rs.
		\end{aligned}
		\end{equation}
		
		\end{proof}

	\section{Global Convergence with weaker assumptions on Self-Concordance} \label{sec:ap_global_nsc}
	We can prove global convergence to a neighborhood of the solution without using self-concordance directly, just by utilizing the following assumptions:
	\begin{mdframed}
		\begin{assumption}[Convexity] \label{as:conv}
			For function $f$ and any $x,h \in \bbE$ holds
			\begin{equation}
			f(x+h) \geq f(x) +\ip{\g(x)} {h} 
			\end{equation}
		\end{assumption} 
		\begin{assumption}[Hessian smoothness, in Hessian norms] \label{as:self_con_glob}
			Objective function $f$ satisfy 
			\begin{equation} \label{eq:self_con_global}
			f(x+h) - f(x) \leq \ip{\g(x)} {h} + \tfrac 1 2 \normM h x ^2 + \tfrac \Lalt 6 \normM h x ^3, \qquad \forall x,h \in \bbE.
			\end{equation}
		\end{assumption}
	\end{mdframed}
	
	\begin{lemma}[One step decrease globally] \label{le:one_step_global}
		Let \Cref{as:self_con_glob} hold and let $\Lalg \geq \Lalt$. Iterates of \ain{} \cref{eq:update} yield function value decrease,
		\begin{equation} \label{eq:one_step_global}
		f(x_{k+1})-f(x_k)
		\leq \begin{cases}
		- \tfrac 1 {2\sqrt \Lalg} \normM {\g(x_k)} {x_k} ^{*\tfrac 32} & \text{if } \normMd {\g(x_k)} {x_k} \geq \tfrac 4 \Lalg \\
		- \tfrac 1 4 \normsMd {\g(x_k)} {x_k} & \text{if } \normMd {\g(x_k)} {x_k} \leq \tfrac 4 \Lalg\\
		- \tfrac {\sqrt {c_1}} {2\sqrt \Lalg} \normM {\g(x_k)} {x_k} ^{*\tfrac 32} & \text{if } \normMd {\g(x_k)} {x_k} \geq \tfrac {4 c_1} \Lalg \text{ and } 0<c_1 \leq 1 
		\end{cases}.
		\end{equation}
	\end{lemma}
	
	Decrease of \Cref{le:one_step_global} is tight up to a constant factor. As far as $\normMd {\g(x_k)} {x_k} \leq \tfrac {4 c_1} {\Lalg}$, we have functional value decrease as the first line of \eqref{eq:one_step_global}, which leads to $\cO \left( k^{-2} \right)$ convergence rate.
	
	We can obtain fast convergence to only neighborhood of solution, because close to the solution, gradient norm is sufficiently small $\normMd {\g(x_k)} {x_k} \leq \tfrac {4c_1} {\Lalg}$ and we get functional value decrease from second line of \eqref{eq:one_step_global}. However, this convergence is slower then $\cO \left( k^{-2} \right)$ for $\normMd {\g(x_k)} {x_k} \approx 0$ and it is insufficient for $\cO(k^{-2})$ rate.
	
	Note that third line generalizes first line; we use it to remove a constant factor gap from the neighborhood of fast local convergence.
	
	\begin{mdframed}
		\begin{theorem}[Global convergence] \label{th:global}
			Let Assumptions \ref{as:conv}, \ref{as:self_con_glob}, \ref{as:level_sets} hold, and constants $c_1, \Lalg$ satisfy $0<c_1\leq 1, \Lalg \geq \Lalt$. For $k$ until $\normMd {\g(x_k)} {x_k} \geq \tfrac {4c_1} {\Lalg}$,  \ain{} has global quadratic decrease, $f(x_k)-f^* \leq \mathcal O \left( \tfrac {\Lalg D^3}{k^2} \right)$. 
		\end{theorem}
	\end{mdframed}
	Note that the global quadratic decrease of \ain{} is only to a neighborhood of the solution. However, once \ain{} gets to this neighborhood, it converges using (faster) local convergence rate (\Cref{th:local}).
	
	\subsection*{Proofs of global convergence without self-concordance}
	Throughout the rest of proofs, we simplify expressions by denoting terms 
    \begin{equation} \label{eq:hg_def}
    \gboxeq{\gk \eqdef \g(x_k)} \qquad \text{and} \qquad \gboxeq{\hk \eqdef \xdiff}, 
    \end{equation}
    for which holds
	\begin{align} \label{eq:hg_relation} 
	\hk = - \alpha_k \h(x_k)^{-1} \gk, \quad \gk = -\tfrac 1 {\alpha_k} \h(x_k) \hk \quad \text {and} \quad \normM {\hk } {x_k} = \sqrt \alpha \normMd {\gk} {x_k}m
	\end{align} 
	and also \gbox{$\G \eqdef\Lalg \normMd \gk {x_k}.$}
	
	\begin{proof}[Proof of \Cref{le:one_step_global}]
	
		We can use \Cref{as:self_con_glob} to obtain
		\begin{align}
		f(x_{k+1})-f(x_k)
		&= f(x_k+\hk) - f(x_k) \nonumber  \\
		&\stackrel {\eqref{eq:self_con_global}}\leq \ip{\g(x_k)} {\hk} + \frac 1 2 \normsM \hk {x_k} + \frac \Lalg 6 \normM \hk {x_k} ^3 \label{eq:self_con_first} \\
		&\stackrel {\eqref{eq:hg_relation}}= -\alpha_k \normsMd \gk {x_k} + \frac 1 2 \alpha_k^2 \normsMd \gk {x_k} + \frac \Lalg 6 \alpha_k^3 \normM \gk {x_k} ^{*3} \nonumber\\
		&= -\alpha_k \normsMd \gk {x_k} \left( 1- \frac 1 2 \alpha_k - \frac \Lalg 6 \alpha_k^2 \normMd \gk {x_k} \right). \label{eq:self_con_last}
		\end{align}
		
		We can simplify bracket in \cref{eq:self_con_last} as
		\begin{align*}
		1- \frac 1 2 \alpha_k - \frac \Lalg 6 \alpha_k^2 \normMd \gk {x_k}
		&= 1 - \frac 1 2 \frac {\sqrt{1+2\G} -1}{\G} - \frac {\G} 6 \left( \frac {\sqrt{1+2\G} -1}{\G} \right) ^2\\
		&= \frac {4\G + 1 - \sqrt{1+2\G}}{6\G}
		\stackrel{\eqref{eq:AG}}\geq \frac 1 2.
		\end{align*}
		Also, we can simplify the other term in \cref{eq:self_con_last} as
		\begin{align*}
		\alpha_k \normsMd \gk {x_k}
		&= \frac {\normMd \gk {x_k} } \Lalg \left( \sqrt{1+ 2\G} -1 \right)\frac{\sqrt{1+ 2\G} +1} {\sqrt{1+ 2\G} +1}
		= \frac {2 \normsMd \gk {x_k} }{\sqrt{1+ 2\G} +1}\\
		&
		\stackrel {\eqref{eq:jensen}} \geq  \frac {2 \normsMd \gk {x_k} } { \sqrt {\G} + 1 + \frac 1 {\sqrt 2} } 
		\geq \frac {2 \normsMd \gk {x_k} } { \sqrt {\G} + 2 }
		\geq \frac { \normsMd \gk {x_k} } { \max \left( \sqrt {\G}, 2 \right) },
		\end{align*}
		and plug these two result into \cref{eq:self_con_last} to obtain first two lines of \eqref{eq:one_step_global}. Third line can be obtained from the first line of \eqref{eq:one_step_global}. For $c_1$ so that $0<c_1\leq 1$ and $x_k$ satisfying $\tfrac {4c_1} \Lalg \leq \normMd {\g(x_k)} {x_k} <\tfrac 4 {\Lalg}$ holds
		
	    \[f(x_{k+1})-f(x_k) \leq -\tfrac 1 4 \normsMd {\g(x_k)} {x_k} \leq -\tfrac {\sqrt {c_1}} {2\sqrt \Lalg} \normM {\g(x_k)} {x_k} ^{*\tfrac 32}.\]
	\end{proof}
	
	\begin{proof}[Proof of \Cref{th:global}]
		As a consequence of convexity (\Cref{as:conv}) and bounded level sets (\Cref{as:level_sets}), we have 
		\begin{align} 
		f(x_{k})-f^{*} 
		&\leq \ip{\gk} {x_k-x_*}
		= \ip{\h(x_k)^{-1/2}\gk} {\h(x_k)^{1/2}(x_k-x_*)}
		\leq \normMd{\gk}{x_k} \normM{x_k-x_*}{x_k} \nonumber\\
		&\leq D\normMd{\gk}{x_k}.
		\end{align}
		Plugging it into \cref{eq:one_step_global} (which holds under \Cref{as:self_con_glob}) and noting that it yields
		\begin{equation}
		f(x_{k+1})-f(x_k)
		\leq 
		- \frac {\sqrt{c_1}} {2\sqrt \Lalg D^{3/2}} \left( f(x_{k})-f^{*} \right) ^{3/2}.
		\end{equation}
		Denote $\tau \eqdef \frac {\sqrt{c_1}} {2\sqrt \Lalg D^{3/2}}$ and $\err_k \eqdef \tau^2 (f(x_k)-f^*) \geq 0$.
		Terms $\err_k$ satisfy recurrence
		\begin{align*}
		\err_{k+1}
		&= \tau^2 (f(x_{k+1})-f^*)
		\stackrel{\eqref{eq:one_step_global}}\leq \tau^2 (f(x_k)-f^*) - \tau^3 (f(x_k)-f^*)^{3/2}
		= \err_k - \err_k^{3/2}.
		\end{align*}
		Because $\err_{k+1} \geq 0$, we have that $\err_k \leq 1$. 
		
		Proposition 1 of \citet{mishchenko2021regularized} shows that the sequence $\left\{ \err_k \right\} _{k=0} ^\infty, 0\leq \err_k \leq 1$ decreases as $\cO \left(\frac 1 {k^2} \right)$. Let $c_2$ be a constant satisfying $\err_k \leq \frac {c_2} {k^2}$ for all $k$ ($c_2 \approx 3$ seems to be sufficient). Finally, fol $k\geq \frac {\sqrt {c_2}}{\tau \sqrt{\varepsilon}} = 2\sqrt{ \frac {c_2\Lalg D^3}{c_1 \varepsilon} } = \cO \left( \sqrt{ \frac {\Lalg D^3}{\varepsilon} } \right)$ we have 
		\begin{equation*}
		f(x_k)-f^*
		= \frac {\err_k} {\tau^2}
		\leq \frac {c_2} {c_1 k^2 \tau^2}
		\leq \varepsilon.
		\end{equation*}
	\end{proof}
	
\end{document}

%% file: all_commands.tex
\usepackage[utf8]{inputenc} % allow utf-8 input
\usepackage[T1]{fontenc}    % use 8-bit T1 fonts
\usepackage{hyperref}       % hyperlinks
\usepackage{url}            % simple URL typesetting
\usepackage{booktabs}       % professional-quality tables
\usepackage{amsfonts}       % blackboard math symbols
\usepackage{nicefrac}       % compact symbols for 1/2, etc.
\usepackage{microtype}      % microtypography
\usepackage{xcolor}         % colors
\usepackage{wrapfig}
\usepackage{natbib}
\usepackage{xargs}                      % Use more than one optional parameter in a new commands
\usepackage{comment}

\usepackage{url, times}

\usepackage{amssymb,amsmath,amscd,amsfonts,amstext,amsthm,bbm, enumerate, dsfont, mathtools}
\usepackage{array}
\usepackage{multicol}
\usepackage{algorithm}
\usepackage{graphicx}
\usepackage{varwidth}
\usepackage{import}
\usepackage{caption,subcaption}
\usepackage{algpseudocode}
\usepackage[none]{hyphenat}
\usepackage{footnote}
\usepackage{cleveref}
\usepackage{makecell}
\usepackage{threeparttable}
\usepackage{booktabs}
\usepackage{tcolorbox}
\usepackage{caption}
\usepackage{subcaption}

\usepackage{scrextend}

\graphicspath{ {images/} }

\newtheorem{theorem}{Theorem}
\newtheorem{proposition}{Proposition}
\newtheorem{lemma}{Lemma}

\newtheorem{assumption}{Assumption}

\newtheorem{definition}{Definition}

\DeclareMathOperator*{\argmin}{argmin}

\newcommand{\R}{\mathbb{R}}

\newcommand{\eqdef}{\stackrel{\text{def}}{=}}
\newcommand{\cO}{{\cal O}}

\usepackage{tablefootnote}
\usepackage{changepage}
\usepackage{stackengine}
\usepackage{pifont}

\usepackage{mdframed}
\mdfsetup{%
	linecolor=white,
	backgroundcolor=blue!5,
}

\definecolor{mydarkgreen}{RGB}{39,130,67}

\definecolor{mydarkred}{RGB}{192,47,25}

\newcommand{\cmark}{{\color{mydarkgreen}\ding{51}}}%
\newcommand{\xmark}{{\color{mydarkred} \ding{55}}}%
\newcommand{\norm}[1]{{\left \| #1 \right\|}_2}

\newcommand{\Lstandard}{{L_{\text{sc}}}}

\newcommand{\Lsemi}{{L_{\text{semi}}}}
\newcommand{\Lstrongly}{{L_{\text{str}}}}
\newcommand{\Lalg}{{L_{\text{est}}}}
\newcommand{\Lalt}{{L_{\text{alt}}}}

\newcommand{\h}{\nabla^2 f}
\newcommand{\g}{\nabla f}

\newcommand{\mI}{\mathbf I}

\newcommand{\normM}[2]{{\left \| #1 \right\|}_{#2}}
\newcommand{\normsM}[2]{{\left \| #1 \right\|}_{#2}^2}
\newcommand{\normMd}[2]{{\left \| #1 \right\|}_{#2}^*}
\newcommand{\normsMd}[2]{{\left \| #1 \right\|}_{#2}^{*2}}

\newcommand{\G}{G_k}
\newcommand{\gk}{g_k}
\newcommand{\hk}{h_k}
\newcommand{\err}{\beta}
\newcommand{\gn}{g_{k+1}} %{\g(x^{k+1})}
\newcommand{\gbox}{\colorbox{lightgray!50}}
\newcommand{\gboxeq}[1]{\text{\colorbox{lightgray!50} {$#1$}}}

\usepackage{theoremref}

\usepackage[colorinlistoftodos,bordercolor=orange,backgroundcolor=orange!20,linecolor=orange,textsize=scriptsize]{todonotes}

\newcommand{\ain}{{\sf\color{red}AICN}}

\newcommand{\sumin}[2]{ \sum \limits_{#1=1}^{#2}}

\newcommand{\norms}[1]{\left \| #1 \right\|^2_2}

\newcommand{\xdiff}{x_{k+1} -x_k}

\newcommand{\ip}[2]{\left\langle #1, #2  \right \rangle}

\def\<#1,#2>{\left\langle #1,#2\right\rangle}

\usepackage{enumitem}
\newlist{myitemize}{itemize}{3}
\setlist[myitemize,1]{label=\textbullet,leftmargin=0in}

\newcommand{\ls}{\left(}
\newcommand{\rs}{\right)}

\newcommand{\lb}{\left\lbrace}
\newcommand{\rb}{\right\rbrace}
\newcommand{\la}{\left\langle}
\newcommand{\ra}{\right\rangle}

\newcommand{\bbE}{\mathbb{E}}